%% file: Journal_Asyn_EXTRA.tex
\documentclass[tran,10pt]{IEEEtran}

\usepackage{hyperref}
\usepackage{cite}
\usepackage{booktabs}
\usepackage{threeparttable}

\input{macros}

\mathtoolsset{showonlyrefs}
\allowdisplaybreaks

\def\Zint{{\mathchoice{\setbox1=\hbox{\sf Z}\copy1\kern-.75\wd1\box1}
{\setbox1=\hbox{\sf Z}\copy1\kern-.75\wd1\box1}
{\setbox1=\hbox{\scriptsize\sf Z}\copy1\kern-.75\wd1\box1}
{\setbox1=\hbox{\scriptsize\sf Z}\copy1\kern-.75\wd1\box1}}}

\makeatletter
\def\hlinewd#1{%
  \noalign{\ifnum0=`}\fi\hrule \@height #1 \futurelet
   \reserved@a\@xhline}
\makeatother
\title{Decentralized Consensus Optimization with Asynchrony and Delays}

\author{\IEEEauthorblockN{Tianyu Wu, Kun Yuan, Qing Ling,  Wotao Yin, and Ali H. Sayed \\}
        \vspace{-0.4cm}

        \thanks{This work was supported in part by NSF grants CCF-1524250, DMS-1317602, ECCS-1407712, and ECCS-1462397, NSF China grant 61573331, and DARPA project N66001-14-2-4029. A short version of this work appeared in the conference publication \cite{wu2016decentralized}.

        T. Wu and W. Yin are with the Department of Mathematics, University of California, Los Angeles, CA90095, USA. K. Yuan and A. H. Sayed are with the Department of Electrical Engineering, University of California, Los Angeles, CA90095, USA. Q. Ling is with the Department of Automation, University of Science and Technology of China, Hefei, Anhui 230026, China. Emails: \{wuty11,kunyuan,wotaoyin,sayed\}@ucla.edu,~qingling@mail.ustc.edu.cn

        }

        }
\renewcommand{\footnoterule}{%
        \kern -1pt
        \hrule width 1.1in height 1.2pt
        \kern 2pt
}

\IEEEoverridecommandlockouts
\usepackage{color}

\begin{document}

\def\helvetica{phvr7t.tfm}
\def\helveticaoblique{phvro7t.tfm}
\def\helveticabold{phvb7t.tfm}
\def\helveticaboldoblique{phvbo7t.tfm}

\font\sfb=\helveticabold
=\helveticaboldoblique
\maketitle

\begin{abstract}
We propose an asynchronous, decentralized algorithm for consensus optimization. The algorithm runs over a network in which the agents communicate with their neighbors and perform local computation.

In the proposed algorithm, each agent can compute and communicate independently at different times, for different durations, with the information it has even if the latest information from its neighbors is not yet available. Such an asynchronous algorithm reduces the time that agents would otherwise waste idle because of communication delays or because their neighbors are slower. It also eliminates the need for a global clock for synchronization.

Mathematically, the algorithm involves both primal and dual variables, uses fixed step-size parameters, and provably converges to the exact solution {under a random agent assumption and both bounded and unbounded delay assumptions}. When running synchronously, the algorithm performs just as well as existing competitive synchronous algorithms such as PG-EXTRA, which diverges without synchronization. Numerical experiments confirm the theoretical findings and illustrate the performance of the proposed algorithm.
%
%
\end{abstract}

%

\section{Introduction and Related Work}\label{sec:introduction}
\input{introduction.tex}
\section{Algorithm Development}
\input{problem.tex}
\input{algorithm.tex}
\section{Convergence Analysis}
\input{operator.tex}
\input{derivation.tex}
\input{analysis.tex}
\section{Numerical Experiments}
\input{numerical.tex}
\input{conclusion.tex}
\bibliographystyle{IEEEbib}
\bibliography{reference,reference1}

\end{document}

%% file: macros.tex
\usepackage{mathtools}
\usepackage[normalem]{ulem} 
\usepackage{subcaption}
\usepackage{enumerate}
\usepackage[ruled,vlined]{algorithm2e}
\usepackage{amsfonts,amsthm}
\usepackage{mathrsfs}
\usepackage{amssymb}
\usepackage{booktabs}
\usepackage{multirow,color}
\usepackage{amsfonts,dsfont}

\newtheorem{remark}{Remark}
\newtheorem{corollary}{Corollary}
\newtheorem{definition}{Definition}

\newcommand{\remove}[1]{{}}
\newcommand{\cut}[1]{}

\newcommand{\nnb}{\nonumber \\}

\newcommand{\eq}[1]{\begin{align}#1\end{align}}
\newcommand{\ba}{\left[ \begin{array}}
\newcommand{\ea}{\\ \end{array} \right]}

\newcommand{\vf}{{\mathbf{f}}}

\newcommand{\vr}{{\mathbf{r}}}
\newcommand{\vs}{{\mathbf{s}}}

\newcommand{\vu}{{\mathbf{u}}}

\newcommand{\vx}{{\mathbf{x}}}
\newcommand{\vy}{{\mathbf{y}}}
\newcommand{\vz}{{\mathbf{z}}}

\newcommand{\vS}{{\mathbf{S}}}

\newcommand{\vZ}{{\mathbf{Z}}}

\newcommand{\cA}{{\mathcal{A}}}
\newcommand{\cB}{{\mathcal{B}}}

\newcommand{\cE}{{\mathcal{E}}}
\newcommand{\cF}{{\mathcal{F}}}
\newcommand{\cG}{{\mathcal{G}}}
\newcommand{\cH}{{\mathcal{H}}}
\newcommand{\cI}{{\mathcal{I}}}

\newcommand{\cL}{{\mathcal{L}}}

\newcommand{\cN}{{\mathcal{N}}}
\newcommand{\cO}{{\mathcal{O}}}
\newcommand{\cP}{{\mathcal{P}}}
\newcommand{\cQ}{{\mathcal{Q}}}

\newcommand{\cS}{{\mathcal{S}}}
\newcommand{\cT}{{\mathcal{T}}}

\newcommand{\cV}{{\mathcal{V}}}

\newcommand{\cZ}{{\mathcal{Z}}}

\newcommand{\RR}{\mathbb{R}}

\newcommand{\EE}{\mathbb{E}}

\newcommand{\St}{{\mathrm{subject~to}}} 
\newcommand{\grad}{{\nabla}}    
\newcommand{\tr}{{\mathrm{tr}}} 

\newcommand{\prox}{\mathbf{prox}}

\DeclareMathOperator*{\argmin}{arg\,min}

\DeclareMathOperator*{\Min}{minimize}

\newcommand{\bc}{\begin{center}}
\newcommand{\ec}{\end{center}}

\newcommand{\bdm}{\begin{displaymath}}
\newcommand{\edm}{\end{displaymath}}

\newcommand{\beq}{\begin{equation}}
\newcommand{\eeq}{\end{equation}}

\newcommand{\bfl}{\begin{flushleft}}
\newcommand{\efl}{\end{flushleft}}

\newcommand{\bt}{\begin{tabbing}}
\newcommand{\et}{\end{tabbing}}

\newcommand{\beqn}{\begin{align}}
\newcommand{\eeqn}{\end{align}}

\newcommand{\beqs}{\begin{align*}} 
\newcommand{\eeqs}{\end{align*}}  

\newtheorem{assumption}{Assumption}

\newtheorem{theorem}{{Theorem}}

\newtheorem{lemma}{{Lemma}}
\newtheorem{proposition}{Proposition}


%% file: introduction.tex
This paper considers a connected network of $n$ agents that cooperatively solve the \textit{consensus optimization} problem
\eq{
&\Min_{x\in \RR^p}\quad \bar{f}(x):=\frac{1}{n}\sum_{i=1}^{n} f_i(x), \nnb
&\text{where}~ f_i(x):=s_i(x) + r_i(x),~i=1,\ldots,n.\label{prob-general}
}
We assume that the functions $s_i$ and $r_i: \RR^p \rightarrow \RR$ are convex \textit{differentiable} and possibly \textit{nondifferentiable} functions, respectively. We call $f_i=s_i+r_i$  a \emph{composite} objective function. Each $s_i$ and $r_i$ are kept private by agent $i=1,2,\cdots,n$, and $r_i$ often serves as the regularization term or the indicator function to a certain constraint on the optimization variable $x \in \mathbb{R}^p$ that is common to all the agents. 
Decentralized  algorithms   rely on agents' local computation, as well as  the information exchange between agents. Such algorithms are generally  robust to failure of critical relaying agents and  scalable with network size.

In decentralized {\color{black}computation, especially with heterogeneous agents or due to processing and communication delays, it can be inefficient or impractical} to synchronize  multiple nodes and links. To see this, let {\color{black}$x^{i, k}\in\RR^p$} be the local variable of agent $i$ at iteration $k$, and let {\color{black}$X^k=[x^{1,k},x^{2,k},\ldots,x^{n,k}]^\top\in\RR^{n\times p}$} collect all local variables, where $k$ is the iteration index.
 In a synchronous implementation, in order to perform an iteration that updates the entire $X^{k}$ to $X^{k+1}$,  all agents will need to wait for the slowest agent or be held back by the slowest communication. {\color{black}In addition, a clock coordinator is necessary for synchronization, which can be expensive and demanding to maintain in a large-scale decentralized network.}
	
{\color{black}Motivated by these considerations, this paper proposes an asynchronous decentralized algorithm where actions by agents are not required to run synchronously. To allow agents to  compute and communicate at different moments, for different durations, the proposed algorithm introduces delays into the iteration --- the update of $X^k$ can rely on delayed information received from neighbors. The information may be several-iteration out of date. Under uniformly bounded (but  arbitrary)  delays and that {the next update is done by a random agent}
	, this paper will show that the sequence $\{X^{k}\}_{k\ge 0}$ converges to a solution to Problem \eqref{prob-general} with probability one.}

What can cause delays? Clearly, communication latency introduces delays. Furthermore, as agents start and finish their iterations independently, one agent may have updated its variables
 while its neighbors are working on their current iterations that still use old (i.e., delayed) copies of those variables; this situation is illustrated in Fig. \ref{fig:delay}. {\color{black}For example,
 	before iteration $3$, agents $1$ and $2$ have finished updating their local variables $x^{1,2}$ and $x^{2,1}$ respectively, but agent $3$ is still relying on delayed neighboring variables $\{x^{1,0}, x^{2,0}\}$, rather than the updated variables $\{x^{1,2},x^{2,1}\}$,  to update $x^{3,3}$.} Therefore, both computation and communication cause delays.


\begin{figure}
        \includegraphics[scale=0.4]{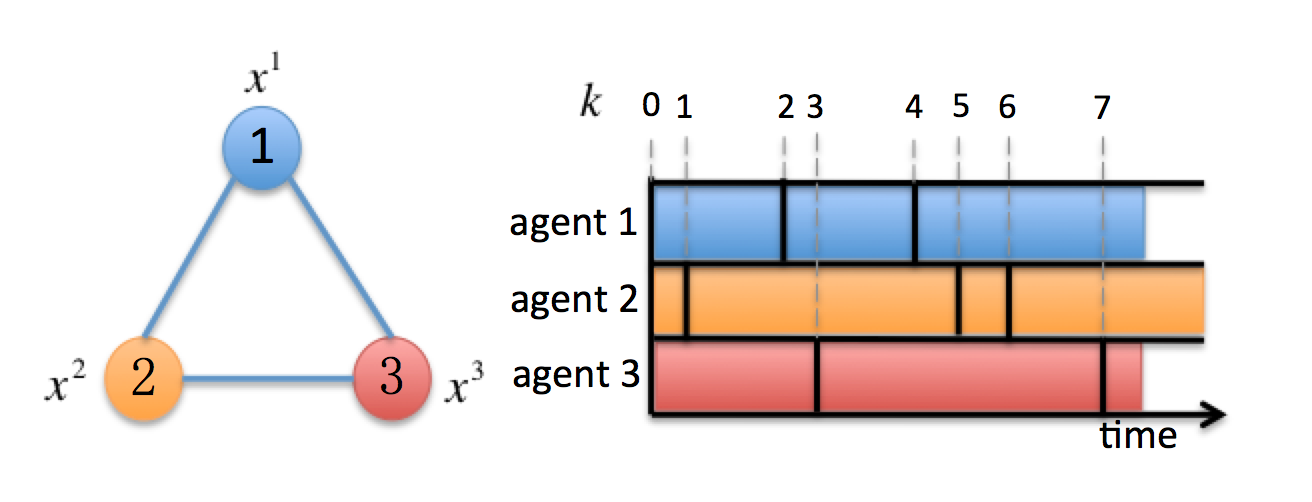}
        \caption{Network and uncoordinated computing.}
        \label{fig:delay}
\end{figure}



\subsection{Relationship to certain synchronous algorithms}
The proposed algorithm, if running synchronously, can be algebraically reduced to PG-EXTRA \cite{shi2015proximal}; they  solve problem \eqref{prob-general} with a fixed step-size parameter and are typically faster than algorithms using diminishing step sizes.
Also, both algorithms generalize EXTRA \cite{shi2014linear}, which only deals with differentiable functions.
{However, divergence (or convergence to wrong solutions) can be observed when one runs EXTRA and PG-EXTRA in the \emph{asynchronous} setting,} 
where the proposed algorithm works correctly. {The proposed algorithm in this paper must use additional dual variables that are associated with edges, thus leading to a moderate cost of updating and communicating the dual variables.}

The proposed algorithm is also very different from decentralized
ADMM \cite{schizas2008consensus,shi2014linear,chang2015multi} except that both algorithms can use fixed
parameters. Distributed consensus methods \cite{nedic2009distributed,yuan2013convergence} that rely on fixed step-sizes can also converge fast, albeit only to {\em approximate}  solutions \cite{nedic2009distributed,yuan2013convergence}.

%

{\color{black}Several useful diffusion strategies \cite{chen2013distributed,chen2015learning,chen2015learning2,sayed2014adaptive,vlaski2015proximal,vlaski2016diffusion} have also been developed for solving {\em stochastic} decentralized optimization problems where realizations of random cost functions are observed at every iteration. To keep continuous learning alive, these strategies also employ fixed step-sizes, and they converge fast to a small neighborhood of the true solution. The diffusion strategies operate on the primal domain, but they can outperform some primal-dual strategies in the stochastic setting due to the influence of gradient noise \cite{towfic2015stability}. These studies focused on synchronous implementations. Here, our emphasis is on asynchronous networks, where delays are present and become critical, and also on deterministic optimization where convergence to the true solution of Problem \eqref{prob-general} is desired.}

%
\vspace{-3mm}
\subsection{Related decentralized algorithms under different settings}
Our setting of asynchrony is different from \textit{randomized single activation}, which is assumed for the randomized gossip algorithm \cite{boyd2006randomized,dimakis2010gossip}. Their setting activates only one edge at a time and does \emph{not} allow any delay. That is, before each activation, computation and communication associated with previous activations must be completed, and only one edge in each neighborhood can  be activated at any time. Likewise, our setting is different from \textit{randomized multi-activation} such as \cite{kar2008sensor,fagnani2008randomized} for consensus averaging, and
 \cite{nedic2011asynchronous,iutzeler2013asynchronous,lorenzo2012decentralized,wei20131,hong2015stochastic,bianchi2016coordinate,zhang2015bi} for consensus optimization, which activate multiple edges each time and still do not allow any delay. These algorithms can be alternatively viewed as synchronous algorithms running in a sequence of  varying subgraphs. Since each iteration waits for the slowest agent or longest communication, a certain coordinator or global clock is needed.

 {\color{black} Our setting is also different from \cite{nedic2015distributed,zhao2015asynchronous,zhao2015asynchronous2,zhao2015asynchronous3}, in which other sources of asynchronous behavior in networks are allowed, such as {different arrival times of data at the agents}, random on/off activation of edges and neighbors, random on/off updates by the agents, and random link failures, etc. Although the results in these works provide notable evidence for the resilience of decentralized algorithms to network uncertainties, they do not consider delays.
}
%

We also distinguish our setting  from the \emph{fixed communication delay} setting \cite{tsianos2011distributed,tsianos2012distributed}, where the information passing through each edge takes a fixed number of iterations to arrive. Different edges can have different such numbers, and agents can compute with only the information they have, instead of waiting. As demonstrated in \cite{tsianos2011distributed}, this  setting can be transferred into  \emph{no communication delay}  by replacing an edge with a chain of dummy nodes. Information passing through a chain of $\tau$ dummy  nodes simulates an edge with a $\tau$-iteration delay. The computation in this setting is still synchronous, so a coordinator or global clock is still needed. \cite{tsianos2011distributed,liu2011distributed} consider \emph{random communication delays} in their setting. However, they are only suitable for consensus averaging, not generalized for the optimization problem \eqref{prob-general}.
	

Our setting is identical to the setting outlined in Section 2.6 of \cite{Peng_2015_AROCK}, the \emph{asynchronous decentralized ADMM}. Our algorithm, however, handles composite functions and avoids solving possibly complicated subproblems.
The recent paper \cite{eisen2016decentralized} also considers both random computation and communication delays. However, as the focus of that paper is a quasi-Newton method, its analysis assumes twice continuously differentiable functions and thus excludes nondifferentiable functions $s_i$ in our problem. In fact, \cite{eisen2016decentralized} solves a different problem and cannot be directly used to solve our Problem \eqref{prob-general} even if its objective functions are smooth. It can, however, solve an approximation problem of \eqref{prob-general} with either reduced solution accuracy or low speed. 

{Our setting is also related with \cite{cannelli2016asynchronous}, in which an asynchronous framework is proposed to solve a composite non-convex optimization problem, where delays and inconsistent reads are allowed. However, the framework in \cite{cannelli2016asynchronous} is designed for parallel computing within one single agent rather than a multi-agent network, which is a fundamental difference with this paper. A summary of various decentralized asynchronous algorithms is listed in Table \ref{literature comparision}. 
}


\begin{table*}
	\setlength{\belowcaptionskip}{-4mm}
	\caption{\small {A comparison among asynchronous decentralized algorithms} \vspace{-0mm}}
	\centering
		\begin{tabular}{|c|c|c|c|c|c|} 
			\hline \scriptsize
			Algorithm & {\scriptsize Synchronization Cost} & {\scriptsize {Allow delays in Recursions}} & {\scriptsize Idle Time} &\scriptsize Cost Function & \scriptsize Exact Convergence$^*$\\
			\hline 
			\scriptsize Gossip\cite{boyd2006randomized,dimakis2010gossip} & \scriptsize Reduced & \scriptsize No & \scriptsize Yes & \scriptsize Average Consensus & \scriptsize Yes\\
			\hline
			\scriptsize Randomized Consensus\cite{kar2008sensor,fagnani2008randomized} & \scriptsize Reduced & \scriptsize No & \scriptsize Yes & \scriptsize Average Consensus & \scriptsize Yes\\
			\hline
			\scriptsize \scriptsize Randomized DGD\cite{nedic2011asynchronous,nedic2015distributed} & \scriptsize Reduced & \scriptsize No & \scriptsize Yes & \scriptsize Aggregate Optimization & \scriptsize No\\
			\hline
			\scriptsize Asynchronous
			 Diffusion\cite{zhao2015asynchronous,zhao2015asynchronous2,zhao2015asynchronous3} & \scriptsize Reduced & \scriptsize No & \scriptsize Yes & \scriptsize Aggregate stochastic Optimization & \scriptsize No\\
			\hline
			\scriptsize Randomized ADMM\cite{iutzeler2013asynchronous,wei20131,zhang2015bi} & \scriptsize Reduced & \scriptsize No & \scriptsize Yes & \scriptsize Aggregate Optimization & \scriptsize Yes\\
			\hline
			\scriptsize Randomized Prox-DGD\cite{hong2015stochastic} & \scriptsize Reduced & \scriptsize No & \scriptsize Yes & \scriptsize Composite Aggregate Optimization & \scriptsize No\\
			\hline
			{\scriptsize Asynchronous Newton\cite{eisen2016decentralized}} & \scriptsize{None} & \scriptsize{Yes} & \scriptsize{No} & \scriptsize{Aggregate Optimization} & \scriptsize{No}\\
			\hline
			{\scriptsize Asynchronous ADMM\cite{Peng_2015_AROCK} } & \scriptsize{None} & \scriptsize{Yes} & \scriptsize{No} & \scriptsize{Aggregate Optimization} & \scriptsize{Yes}\\
			\hline
			\scriptsize \textbf{\scriptsize Asynchronous Prox-DGD \eqref{dpg-asyn}} & \scriptsize \textbf{None} & \scriptsize \textbf{Yes} & \scriptsize \textbf{No} & \scriptsize \textbf{\scriptsize Composite Aggregate Optimization} & \scriptsize \textbf{No}\\
			\hline
			\textbf{\scriptsize Proposed Alg. \ref{alg:asyn}} & \scriptsize \textbf{None} & \scriptsize \textbf{Yes} & \scriptsize \textbf{No} & \scriptsize \textbf{Composite Aggregate Optimization} & \scriptsize \textbf{Yes}\\
			\hline 
			
			
		\end{tabular}
		\label{literature comparision}
		\begin{tablenotes}
			\item[*]$^*$ {Convergence to the exact optimal solution when step size is constant.}
		\end{tablenotes}
\end{table*}


\vspace{-1mm}
\subsection{Contributions}
This paper introduces an asynchronous, decentralized algorithm for problem \eqref{prob-general} that provably converges to an optimal solution assuming that the next update is performed by a random agent
and that communication is subject to arbitrary {and possibly unbounded} delays. If running synchronously, the proposed algorithm is as fast as the competitive PG-EXTRA algorithms except, for asynchrony, the proposed algorithm involves updating and communicating additional edge variable. 

Our asynchronous setting is considerably less restrictive than the settings  under which existing {non-synchronous} or {non-deterministic} decentralized algorithms are proposed. In our setting,  the computation and communication of agents are uncoordinated. A global clock is not needed.

{\color{black}
Technical contributions are also made. We borrow ideas from monotone operator theory and primal-dual operator splitting to derive the proposed algorithms in compact steps. To establish convergence under delays, we cannot follow the existing analysis of PG-EXTRA \cite{shi2015proximal}; instead, motivated by \cite{Peng_2015_AROCK,peng2016coordinate}, a new non-Euclidean metric is introduced to absorb the delays.
%
In particular, the foundation of the analysis is not the monotonicity of the objective value sequence. Instead, the proposed approach establishes monotonic conditional expectations of certain distances to the solution. We believe this new analysis can extend to a general set of primal-dual algorithms beyond decentralized optimization.
}

\vspace{-2mm}
\subsection{Notation}\label{sec:notation}
Each agent $i$ holds a local variable $x^i \in \RR^p$, whose value at iteration $k$ is denoted by $x^{i,k}$. {We introduce variable $X$ to stack all local variables $x^{i}$:
\eq{
        X := \ba{ccc}
        \rule[0.8mm]{0.5cm}{0.3mm} & \left(x^{1}\right)^\top & \rule[0.8mm]{0.5cm}{0.3mm} \vspace{1mm} \\
        & \vdots & \\
        \rule[0.8mm]{0.5cm}{0.3mm} & \left(x^{n}\right)^\top & \rule[0.8mm]{0.5cm}{0.3mm}
        \ea \in \RR^{n \times p}.\label{notation-vx}
}}

In this paper, we denote the $j$-th row of a matrix $A$ as $A_j$. Therefore, for $X$ defined in \eqref{notation-vx}, it holds that $X_i=(x^i)^\top$, which indicates that the $i$th row of $X$ is the local variable $x^i \in \RR^p$ of agent $i$. Now we define functions
\eq{\label{notation-vf}
        \vs(X):= \sum_{i=1}^{n} s_i(x^i),\quad \vr(X):= \sum_{i=1}^{n} r_i(x^i),
}
as well as
\eq{\label{fx}
\vf(X) := \sum_{i=1}^{n}f_i(x^i)=\vs(X) + \vr(X).
} The gradient of $\vs(X)$ is defined as
\eq{\label{notation-grad-s}
\grad \vs(X) := \ba{ccc}
\rule[0.8mm]{0.5cm}{0.2mm} & \big(\grad s_1(x^1 )\big)^\top & \rule[0.8mm]{0.5cm}{0.2mm} \vspace{1mm} \\
& \vdots & \\
\rule[0.8mm]{0.5cm}{0.2mm} & \big(\grad s_n(x^n)\big)^\top & \rule[0.8mm]{0.5cm}{0.2mm}
\ea \in \RR^{n\times p}.
}
The inner product on $\RR^{n\times p}$ is defined as
$
\langle
X,\widetilde{X}\rangle:=\tr(X^\top\widetilde{X})\hspace{-1mm}=\hspace{-1mm}\sum_{i=1}^n(x^i)^\top \widetilde{x}^i,
$
where $X, \widetilde{X}$ are arbitrary matrices.

%% file: problem.tex

Consider a strongly connected network $\cG=\{\cV, \cE\}$ with  agents $\cV=\{1,2,\cdots, n\}$ and   \emph{undirected} edges $\cE=\{1,2,\cdots,m\}$. By convention, all edges $(i,j)\in\cE$ obey $i<j$. To each edge $(i,j)\in \cE$, we assign a weight $w_{ij}> 0$, which is used by agent $i$ to scale the data $x^j$ it receives from agent $j$. Likewise, let $w_{ji}=w_{ij}$ for agent $j$. If $(i,j)\notin \cE$, then $w_{ij}=w_{ji}=0$. {For each  $i$, we let
$
	\cN_i := \{j|(i,j)\in \cE \mbox{ or } (j,i)\in \cE\} \cup \{i\}
$
denote the set of neighbors of agent $i$ (including agent $i$ itself). We also let 
$
	\cE_i := \{(i,j)|(i,j)\in \cE \mbox{ or } (j,i)\in \cE\}
$
denote the set of all edges connected to $i$. Moreover, we also assume that there exists at least one agent index $i$ such that $w_{ii}>0$.}


%
%

Let $W=[w_{ij}]\in \RR^{n\times n}$ denote the weight matrix, which is  symmetric and doubly stochastic. Such $W$ can be generated through the maximum-degree or Metropolis-Hastings rules \cite{sayed2014adaptation}. It is easy to verify that
$\mathrm{null}\{I-W\}=\mathrm{span}\{\mathds{1}\}$.

The proposed algorithm involves a matrix $V$, which we now define. Introduce the diagonal matrix $D\in \RR^{m\times m}$ with diagonal entries $D_{e,e}=\sqrt{w_{ij}/2}$ for each edge $e=(i,j)$. Let $C=[c_{ei}]\in \RR^{m\times n}$ be the incidence matrix of $\cG$:
%
{\color{black}\eq{
\label{inci-mat}
c_{ei} =
\begin{cases}
       \begin{array}{ll}
               +1,\ &\text{$i$ is the lower indexed agent connected to $e$}, \\
               -1,\ &\text{$i$ is the higher indexed agent connected to $e$}, \\
               0,\ &\text{otherwise}.
       \end{array}
\end{cases}
}}With incidence matrix \eqref{inci-mat},
we further define
\eq{
V:=DC \in \RR^{m\times n}. \label{V-definition}
}
as the  scaled incidence matrix. It is easy to verify:
\begin{proposition}[Matrix factorization identity] \label{lm-matrix-factorization}
	When $W$ and $V$ is generated according to the above description, we have $V^\top V=(I-W)/2.$

\end{proposition}
\subsection{Problem formulation}
\label{sec-primal-dual algrithm}
Let us reformulate Problem~\eqref{prob-general}.
First, it is equivalent to
\eq{\label{prob-genral-local}
        \Min_{x^1,\cdots, x^n\in \RR^p} \quad& \sum_{i=1}^{n} s_i(x^i) + \sum_{i=1}^{n} r_i(x^i), \nnb
        \St \quad & x^1=x^2=\cdots=x^n.
}
Since $\mathrm{null}\{I-W\}=\mathrm{span}\{\mathds{1}\}$ and recall the definition of $X$ in \eqref{notation-vx}, Problem \eqref{prob-genral-local} is equivalent to
\eq{\label{prob-compact-I-W}
        \Min_{X\in \RR^{n\times p}}&\quad \vs(X) + \vr(X), \nnb
        \St&\quad (I-W) X = 0.
}
{\color{black}
By Proposition \ref{lm-matrix-factorization}, we have $(I-W)X=0 \Rightarrow V^\top V X=0 \Rightarrow X^\top V^\top V X =0 \Rightarrow V X = 0$. On the other hand, $V X = 0 \Rightarrow V^\top V X = 0 \Rightarrow (I - W)X = 0$. 
Therefore, we conclude that
$
(I - W)X = 0 \Leftrightarrow VX = 0.
$
Therefore, Problem \eqref{prob-compact-I-W} is further equivalent to
\eq{\label{prob-compact-V}
        \Min_{\vx\in \RR^{n\times p}}&\quad \vs(X) + \vr(X), \nnb
        \St&\quad V X = 0.
}
{\color{black}
Now we denote 
\eq{
	\label{notation-vy}
	Y := \ba{ccc}
	\rule[0.8mm]{0.5cm}{0.3mm} & \left(y^{1}\right)^\top & \rule[0.8mm]{0.5cm}{0.3mm} \vspace{1mm} \\
	& \vdots & \\
	\rule[0.8mm]{0.5cm}{0.3mm} & \left(y^{m}\right)^\top & \rule[0.8mm]{0.5cm}{0.3mm}
	\ea \in \RR^{m \times p}
}}as the dual variable, Problem \eqref{prob-compact-V} can be reformulated into the saddle-point problem
\begin{equation}
\max_{Y\in\RR^{m\times p}}\min_{X\in\RR^{n\times p}} \vs(X)+\vr(X)+\frac{1}{\alpha}\langle Y,VX\rangle,\label{prob:saddle}
\end{equation}
where $\alpha>0$ is a constant parameter. {Notice that a similar formulation using the incidence and Laplacian matrices was employed in \cite{towfic2015stability} to derive primal-dual distributed optimization strategies over networks.}

%% file: algorithm.tex
\subsection{Synchronous algorithm}
\label{sec:syn-extra}

Problem~\eqref{prob:saddle} can be solved iteratively by the  primal-dual algorithm that is adapted from~\cite{condat2013primal}\cite{vu2013splitting}: 
\begin{equation}
\begin{cases}
X^{k\hspace{-0.2mm}+\hspace{-0.2mm}1}\hspace{-0.8mm}=\hspace{-0.8mm}\prox_{\alpha\vr}[X^k\hspace{-1mm}-\hspace{-1mm}\alpha\nabla\vs(X^k)\hspace{-1mm}-\hspace{-1mm}V^\top(2Y^{k+1}\hspace{-1mm}-\hspace{-1mm}Y^k)],\\
Y^{k+1}\hspace{-0.6mm}=\hspace{-0.6mm}Y^k+VX^k, \label{alg:cv}
\end{cases}
\end{equation}
where the proximal operator is defined as
{\color{black}
\eq{\label{prox-operator}
	\prox_{\alpha\vr}(U):=\argmin_{X\in \RR^{n\times p}} \big\{\vr(X)+\frac{1}{2\alpha}\|X-U\|^2_{\rm F}\big\}.
}}Next, in the $X$-update, we eliminate  $Y^{k+1}$ and use $I-2V^\top V=W$ to arrive at:
\begin{equation}
\begin{cases}
X^{k+1}=\prox_{\alpha\vr}[WX^k\hspace{-0.8mm}-\hspace{-0.8mm}\alpha\nabla\vs(X^k)\hspace{-0.8mm}-\hspace{-0.8mm}V^\top Y^k],\\
Y^{k+1}=Y^k+VX^k,\label{alg:cv2}\\
\end{cases}
\end{equation}
which computes $(Y^{k+1},X^{k+1})$ from $(Y^{k},X^{k})$. Algorithm \eqref{alg:cv2} is essentially equivalent to PG-EXTRA developed in \cite{shi2015proximal}. {\color{black}Algorithm \eqref{alg:cv2} can run in a decentralized manner. To do so, we associate each row of the dual variable $Y$ with an edge $e=(i,j)\in \cE$, and for simplicity we let agent $i$ store and update the variable $y^e$ (the choice is arbitrary). We also define
\eq{
	\cL_i:=\{e=(i,j)\in \cE,\ \forall\,j>i\}, \label{Li}
}
as the index set of dual variables that agent $i$ needs to update.

Recall $V$ is the scaled incidence matrix defined in \eqref{V-definition}, while $W$ is the weight matrix associated with the network. The calculations of $VX$, $V^\top Y$ and $W X$ require communication, and  other operations are local. For agent $i$, it updates its local variables $x^{i,k+1}$ and $\{y^{e,k+1}\}_{e\in \cL_i}$ according to
\begin{equation}\label{syn-decentralized}
\begin{cases}
\displaystyle{x^{i,k+1}\hspace{-0.5mm}=\hspace{-0.5mm}\prox_{\alpha r_i}\hspace{-0.5mm}\Big(\hspace{-1.2mm}\sum_{j\in \mathcal{N}_i} w_{ij}x^{j,k}\hspace{-0.3mm}-\hspace{-0.3mm}\alpha\nabla s_i(x^{i,k})\hspace{-0.5mm}-\hspace{-0.5mm}\sum_{e\in \mathcal{E}_i} v_{e i}y^{e, k}\Big),}\\
\displaystyle{
	y^{e,k+1} \hspace{-0.5mm}=\hspace{-0.5mm} y^{e,k} \hspace{-0.5mm}+\hspace{-0.5mm} \big(v_{ei} x^{i,k} \hspace{-0.5mm}+\hspace{-0.5mm} v_{ej} x^{j,k}\big),\ \forall\, e\in \cL_i,}\\
\end{cases}
\end{equation}
where $v_{ei}$ and $v_{ej}$ are the $(e,i)$-th and $(e,j)$-th entries of the matrix $V$, respectively. }

Algorithm~\ref{alg:syn} implements recursion~\eqref{syn-decentralized} in the synchronous fashion, which requires two \emph{synchronization barriers} in each iteration $k$. The first one holds computing until an agent  receives all necessary input; after the agent finishes computing its variables,  the second barrier  prevents it from sending out information until all of its neighbors  finish their computation (otherwise, the information intended   for iteration $k+1$ may arrive at a neighbor too early).

\begin{algorithm}
	\SetKwInOut{Input}{Input}
	\Input{Starting point $\{x^{i,0}\},\{y^{e,0}\}$. Set counter $k=0$\;}
	\While{all agents $i\in\cV$ in parallel }{
		\vspace{1mm}
		Wait until  $\{x^{j,k}\}_{j \in \mathcal{N}_i}$ and  $\{y^{e, k}\}_{e\in \cE_i}$ are received;\\[2pt]
		
		Update $x^{i,k+1}$ according to \eqref{syn-decentralized};\\[2pt]
		Update $\{y^{e,k+1}\}_{e\in \cL_i}$ according to \eqref{syn-decentralized};\\[2pt]
		
		Wait until all neighbors finish computing;\\[2pt]
		Set $k \leftarrow k+1$;\\[2pt]
		Send out $\hspace{-0.5mm}x^{i,k+1}\hspace{-0.5mm}$ and $\hspace{-0.5mm}\{y^{e,k+1}\}_{e\in \cL_i}$ to neighbors\;
	}
	\caption{Synchronous  algorithm based on~\eqref{syn-decentralized}}
	\label{alg:syn}\end{algorithm}

\vspace{-6mm}
\subsection{Asynchronous algorithm}
\label{sec:aysn-extra}

In the asynchronous setting, each agent computes and communicates independently without any coordinator. Whenever an arbitrary agent finishes a round of its variables' updates, we let the iteration index $k$ increase by $1$ (see Fig. \ref{fig:delay}). As discussed in Sec. \ref{sec:introduction}, both communication latency and uncoordinated computation result in delays. Let $\tau^k \in \RR^n_+$ and $\delta^k\in \RR^m_+$ be vectors of delays at iteartion $k$. We define $X^{k-\tau^k}$ and $Y^{k-\delta^k}$ as delayed primal and dual variables occurring at iteration $k$:
\eq{
X^{k-\tau^k} &:=
\ba{ccc}
\rule[0.8mm]{0.5cm}{0.3mm} & \left(x^{1, k-\tau^k_1}\right)^\top & \rule[0.8mm]{0.5cm}{0.3mm} \vspace{1mm} \\
& \vdots & \\
\rule[0.8mm]{0.5cm}{0.3mm} & \left(x^{n, k - \tau^k_n}\right)^\top & \rule[0.8mm]{0.5cm}{0.3mm}
\ea \in \RR^{n \times p},\\
Y^{k-\delta^k} &:=
\ba{ccc}
\rule[0.8mm]{0.5cm}{0.3mm} & \left(y^{1, k-\delta^k_1}\right)^\top & \rule[0.8mm]{0.5cm}{0.3mm} \vspace{1mm} \\
& \vdots & \\
\rule[0.8mm]{0.5cm}{0.3mm} & \left(y^{m, k - \delta^k_m}\right)^\top & \rule[0.8mm]{0.5cm}{0.3mm}
\ea \in \RR^{m \times p},\nonumber
}
where $\tau_j^k$ is the $j$-th element of $\tau_k$ and $\delta_e^k$ is the $e$-th element of $\delta^k$.
In the asynchronous setting, recursion \eqref{alg:cv2} is calculated with delayed variables, i.e.,
\eq{\label{asyn-compact-1}
\begin{cases}
	\widetilde{X}^{k+1}\hspace{-1mm}=\hspace{-1mm}\prox_{\alpha\vr}[W\hspace{-0.5mm}X^{k-\tau^k}\hspace{-1.8mm}-\hspace{-0.8mm}\alpha\nabla\vs(X^{k-\tau^k})\hspace{-0.8mm}-\hspace{-0.8mm}V^\top Y^{k-\delta^k}],\\
	\widetilde{Y}^{k+1}\hspace{-1mm}=\hspace{-1mm}Y^{k-\delta^k}+V X^{k-\tau^k}.\\
\end{cases}
}
Suppose agent $i$ finishes update $k+1$. To guarantee convergence, instead of letting $x^{i,k+1}=\widetilde{x}^{i,k+1}$ and $y^{e,k+1}=\widetilde{y}^{e,k+1}$ directly, we propose a relaxation step
\eq{\label{agent-i-update}
\begin{cases}
	\begin{array}{l}
		x^{i,k+1} = x^{i,k} + \eta_i \big( \widetilde{x}^{i,k+1} - x^{i,k-\tau^k_i} \big), \\
		y^{e,k+1} = y^{e,k} + \eta_i \big( \widetilde{y}^{e,k+1} - y^{e,k-\delta^k_e} \big),\ \forall e\in \cL_i.
	\end{array}
\end{cases}
}
where $\widetilde{x}^{i,k+1} - x^{i,k-\tau^k_i}$ and $\widetilde{y}^{e,k+1} - y^{e,k-\delta^k_e}$ behave as updating directions, and $\eta_i \in (0,1)$ behaves as a step size. To distinguish from the step size $\alpha$, we call $\eta_i$ the relaxation parameter of agent $i$.
Its value depends on how out of date agent $i$ receives information from its neighbors. Longer delays require a smaller $\eta_i$, which leads to slower convergence. Since the remaining agents have not finished their updates yet, it holds that
\eq{\label{agent-j-unchanged}
\begin{cases}
	\begin{array}{rl}
		x^{j,k+1} = x^{j,k} & \forall j\neq i,\\
		y^{e,k+1} = y^{e,k} & \forall e \notin \cL_i.
	\end{array}
\end{cases}
}
To write \eqref{agent-i-update} and \eqref{agent-j-unchanged} compactly, we let $S_{\rm p}^i: \RR^{n\times p}\to \RR^{n\times p}$ be the primal selection operator associated with agent $i$. For any matrix $A\in \RR^{n\times p}$, $[S_{\rm p}^i(A)]_j = A_j$ if $j=i$; otherwise $[S_{\rm p}^i(A)]_j = 0$,
where $[S_{\rm p}^i(A)]_j$ and $A_j$ denote the $j$-th row of matrix $S_{\rm p}^i(A)$ and $A$, respectively. Similarly, we also let $S_{\rm d}^{\cL_i}: \RR^{m\times p}\to \RR^{m\times p}$ be the dual selection operator associated with agent $i$. For any matrix $B\in \RR^{m\times p}$, $[S_{\rm d}^{\cL_i}(B)]_e = B_e$ if $e\in \cL_i$; otherwise $[S_{\rm d}^{\cL_i}(B)]_e = 0$.
With the above defined selection operators, \eqref{agent-i-update} and \eqref{agent-j-unchanged} are equivalent to
\eq{\label{asyn-compact-2}
\begin{cases}
	\begin{array}{l}
		X^{k+1} = X^k + \eta_i S_{\rm p}^i \big( \widetilde{X}^{k+1} - X^{k-\tau^k} \big),\\
		Y^{k+1} = Y^k + \eta_i S_{\rm d}^{\cL_i} \big( \widetilde{Y}^{k+1} - Y^{k-\delta^k} \big).
	\end{array}
\end{cases}
}
Recursions \eqref{asyn-compact-1} and \eqref{asyn-compact-2} constitute the asynchronous algorithm.
%
%
%

{
Similar to the synchronous algorithm, the asynchronous recursion \eqref{asyn-compact-1} and \eqref{asyn-compact-2} can also be implemented in a decentralized manner. When agent $i$ is activated at iteration $k$:
\begin{align}
\hspace{-2mm}\begin{cases}
\text{Compute:}\\
\displaystyle \widetilde{x}^{i,k\hspace{-0.3mm}+\hspace{-0.3mm}1}\hspace{-1.2mm}=\hspace{-1.2mm} \prox_{\alpha r_i}\hspace{-0.8mm}\Big(\hspace{-1.5mm}\sum_{j\in\cN_i} \hspace{-2mm} w_{ij}x^{j,k\hspace{-0.3mm}-\hspace{-0.3mm}\tau_j^k}\hspace{-1.5mm}-\hspace{-0.8mm}\alpha \hspace{-0.5mm}\nabla\hspace{-0.3mm} s_i(x^{i,k\hspace{-0.3mm}-\hspace{-0.3mm}\tau_i^k})\hspace{-1.1mm}-\hspace{-2.1mm}\sum_{e \in\cE_i}\hspace{-1.5mm}v_{e i}y^{e,k\hspace{-0.3mm}-\hspace{-0.3mm}\delta_{e}^k}\hspace{-0.8mm}\Big)\\
\widetilde{y}^{e,k+1}\hspace{-1.2mm}=\hspace{-1.2mm}y^{e,k-\delta_e^k} + \big( v_{ei}x^{i,k-\tau_i^k}+  v_{ej}x^{j,k-\tau_j^k} \big),~ \forall e\in\cL_i;\\[5pt]
\text{Relaxed updates:} \\
x^{i,k+1}=x^{i,k}+\eta_i\left(\widetilde{x}^{i,k+1}-x^{i,k-\tau_i^k} \right), \vspace{2mm} \\
y^{e,k+1}=y^{e,k}+\eta_i \left(\widetilde{y}^{e, k+1}-y^{e,k-\delta_e^k} \right), ~ \forall e\in\cL_i,
\end{cases}\label{update_asyn}\end{align}
The entries of $X, Y$ not held by agent $i$ remain unchanged
from $k$ to $k+1$.
Algorithm~\ref{alg:asyn} summarizes the asynchronous updates.}
\begin{algorithm}
\SetKwInOut{Input}{Input}
\Input{Starting point $\{x^{i,0}\},\{y^{e,0}\}$. Set counter $k=0$;}
\While{each agent $i$ asynchronously}{
  Compute per~\eqref{update_asyn} using the information it has \\ \hspace{3mm} available\;
  Send out $\hspace{-0.5mm}x^{i,k+1}\hspace{-0.5mm}$ and $\hspace{-0.5mm}\{y^{e,k+1}\}_{e\in \cL_i}$ to neighbors\;
 }
 \caption{Asynchronous  algorithm based on~\eqref{update_asyn}}
\label{alg:asyn}\end{algorithm}

\vspace{-4mm}
{
\subsection{Asynchronous proximal decentralized gradient descent}
Sec. \ref{sec:aysn-extra} has presented a decentralized primal-dual algorithm for problem \eqref{prob-general}. {\color{black}In \cite{vlaski2015proximal,vlaski2016diffusion,chen2012fast}, proximal gradient descent algorithms are derived to solve stochastic composite optimization problems in a distributed manner by networked agents. Using techniques similar to those developed in \cite{vlaski2015proximal,vlaski2016diffusion,chen2012fast} we can likewise derive a proximal variant for the consensus gradient descent algorithm in \cite{nedic2009distributed}; see \eqref{dpg} below.
Following Sec. \ref{sec:aysn-extra}, its asynchronous version will also be developed; see \eqref{dpg-asyn} below.  
}

To derive the synchronous algorithm, we penalize the constraints of Problem \eqref{prob-compact-V} (which is equivalent to  Problem \eqref{prob-general}) and obtain an auxiliary problem:
\eq{
\Min_{X\in \RR^{n\times p}} \quad \vs(X) + \vr(X) + \frac{1}{\alpha}\|VX\|^2,\label{penalty-consensus}
}
where $\alpha > 0$ is a penalty constant.
By $V^\top V = \frac{1}{2}(I - W)$ in Proposition \ref{lm-matrix-factorization}, Problem \eqref{penalty-consensus} is exactly
\eq{
        \Min_{X\in \RR^{n\times p}} \quad \underbrace{\vs(X) + \frac{1}{2 \alpha}X^\top (I-W) X}_{\rm smooth} + \vr(X),\label{penalty-consensus-I-W}
}
which has the smooth-plus-proximable form. Therefore,  applying
 the proximal gradient method with step-size $\alpha$ yields the iteration:
 \eq{\label{prox-dgd-compact}
\vx^{k+1} & = \prox_{\alpha \vr} \left( X^{k} - \alpha [\grad \vs(X^k) - \frac{1}{\alpha}(I - W) X^k] \right) \nnb
& = \prox_{\alpha \vr} \left(W X^{k} - \alpha \grad \vs(X^k) \right),
}
Recursion \eqref{prox-dgd-compact} is ready for decentralized implementation --- each agent $i\in \cV$ performs
\eq{\label{dpg}
        x^{i,k+1} = \prox_{\alpha r_i} \Big( \sum_{j\in \cN_i} w_{ij} x^{j,k} - \alpha \grad s_i(x^{i,k})\Big).
}
Since algorithm \eqref{dpg} solves the penalty problem \eqref{penalty-consensus-I-W} instead of the original problem \eqref{prob-general}, one must reduce $\alpha$ during the iterations or use a small value. {\color{black}Recursion \eqref{dpg} is similar to the proximal diffusion gradient descent algorithm derived in \cite{vlaski2015proximal,vlaski2016diffusion} in which each agent first applies local proximal gradient descent and then performs weighted averaging with neighbors. {According to \cite{sayed2014adaptive,sayed2014adaptation}, the diffusion strategy removes the asymmetry problem that can cause consensus implementations to become unstable for stochastic optimization problems.}}

%
%

The asynchronous version of \eqref{dpg} allows delayed variables and adds relaxation to \eqref{dpg}; each agent $i$ performs
\eq{\label{dpg-asyn}
        \begin{cases}
                \begin{array}{rl}
                        \widetilde{x}^{i,k+1}\hspace{-3mm} &=\hspace{-1mm} \prox_{\alpha r_i} \hspace{-1mm}\left( \sum_{j\in \cN_i}\hspace{-1mm} w_{ij} x^{j,k - \tau_j^k} \hspace{-1mm}-\hspace{-0.8mm} \alpha \grad s_i(x^{i,k-\tau^k_i})\right), \\
                        x^{i,k+1} \hspace{-3mm}&=\hspace{-1mm} x^{i,k} + \eta_i \left( \widetilde{x}^{i,k+1} - x^{i,k-\tau^k_i} \right).
                \end{array}
        \end{cases}
}
Its convergence follows from treating the proximal gradient iteration as a fixed-point iteration and applying results from \cite{Peng_2015_AROCK}. {Unlike Algorithm \ref{alg:asyn}, the asynchronous algorithm \eqref{dpg-asyn} uses only primal variables, so it is easier to implement. 
Like all other distributed gradient descent algorithms \cite{nedic2009distributed,sayed2014adaptive,vlaski2015proximal,vlaski2016diffusion,chen2012fast}, however,  algorithm \eqref{dpg} and its asynchronous version \eqref{dpg-asyn} must use diminishing step sizes to guarantee exact convergence, which causes slow convergence.}
%


}

%% file: operator.tex
\subsection{Preliminaries}\label{sec:op}
{Our convergence analysis is based on the theory of nonexpansive operators, as used in e.g.,~\cite{bianchi2016coordinate}.}

{In this subsection we first present several definitions, lemmas and theorems that underlie the convergence analysis. }

{
	First we introduce a new symmetric matrix 
	\begin{equation}\label{metric}
	G:=\begin{bmatrix}
	I_n&V^\top\\
	V&I_m
	\end{bmatrix} \in \RR^{m+n}. 
	\end{equation} 
	As we will discuss later, $G$ is an important auxiliary matrix that helps establish the convergence properties of the synchronous Algorithm~\ref{alg:syn} and asynchronous Algorithm~\ref{alg:asyn}. Meanwhile, it also determines the range of step size $\alpha$ that enables the algorithms to converge. {\color{black}Recall that the weight matrix $W$ associated with the network is symmetric and doubly stochastic. Besides, there exists at least one agent $i$ such that $w_{ii}>0$. Under these conditions, it is shown in \cite{sayed2014adaptation} that $W$ has an eigenvalue $1$ with multiplicity $1$, and all other eigenvalues are strictly inside the unit circle. With such $W$, we can show that $G$ is positive definite.}
}

\begin{lemma} It holds that $G \succ 0$.
\end{lemma}
\begin{proof}
	{\color{black}According to the Schur Complement condition on positive definiteness, we known
		$
			G \succ 0 \Longleftrightarrow I \succ 0, I - V^\top V \succ 0.
		$
		Recall Proposition \ref{lm-matrix-factorization} and $\lambda_{\min}(W)>-1$, it holds that 
		$I - V^\top V = \frac{1}{2}(I + W) \succ 0,$
		which proves $G \succ 0$.
	}	
\end{proof}

{
	To analyze the convergence of the proposed Algorithms~\ref{alg:syn} and \ref{alg:asyn}, we still need a classic result from the nonexpansive operator theory. In the following we provide related definitions and preliminary results.}
\begin{definition}
	Let $\cH$ be a finite dimensional vector space (e.g. $\RR^p\mbox{ or }\RR^{n\times p}$) equipped with a certain inner product $\langle\cdot,\cdot\rangle$ and its induced norm $\|\vx\|=\sqrt{\langle\vx,\vx\rangle},\forall\vx\in\cH$. {\color{black}An operator $\cP:\cH\to\cH$} is called nonexpansive if $$\|\cP\vx-\cP\vy\|\leq\|\vx-\vy\|,\forall\vx,\vy\in\cH.$$
\end{definition}

\begin{definition}\label{defi-fixed-point}
	For any operator $\cP:\cH\to\cH$ (not necessarily nonexpansive), we define 
	$\mathbf{Fix} \cP:=\{\vx\in\cH|\vx=\cP\vx\}$
	to be the set of fixed points of operator $\cP$.
\end{definition}

{
\begin{definition}\label{defi-ave-operator}
	When the operator $\cP$ is nonexpansive and $\beta\in(0,1)$, the combination operator
	$$\cQ := (1-\beta)I+\beta\cP$$ 
	is defined as the \emph{averaged operator} associated with $\cP$.
\end{definition}}
\noindent {With Definition \ref{defi-fixed-point} and \ref{defi-ave-operator}, it is easy to verify $\mathbf{Fix}\cQ = \mathbf{Fix}\cP$.}

{The next classical theorem states the convergence property of an averaged operator.}
\begin{theorem}[KM iteration~\cite{krasnosel1955two,bauschke2011convex}]\label{thm:cvg}
	Let $\cP:\cH\to\cH$ be a nonexpansive operator and $\beta \in (0,1)$. Let $\cQ$ be the averaged operator associated with $\cP$. Suppose that the set $\mathbf{Fix} \cQ$ is nonempty. From any starting point $\vz^0$, the iteration 
	\eq{\vz^{k+1} =\cQ \vz^k =(1-\beta)\vz^k+\beta\cP\vz^{k}}
	produces a sequence $\{\vz^k\}_{k\geq 0}$ converging to a point $\vz^* \in \mathbf{Fix} \cQ$ (and hence also to a point $\vz^* \in \mathbf{Fix} \cP$).  
\end{theorem}

For the remainder of the paper, we consider a specific Hilbert space $\mathcal{H}$ for the purpose of analysis.
\begin{definition}\label{defi:4}
	Define $\cH$ to be the Hilbert space $\RR^{(n+m)\times p}$ endowed with the inner product {$\langle Z,\widetilde{Z}\rangle_A:=\tr(Z^\top A\widetilde{Z})$} and the norm $\|Z\|_A:=\sqrt{\langle Z,Z \rangle_A}$, {where $A\in \RR^{(n+m)\times (n+m)}$ is a positive definite matrix, and $Z \in \RR^{(n+m)\times p}$.}
\end{definition}

%% file: derivation.tex
\subsection{{Convergence of Algorithm~\ref{alg:syn}}}\label{sec:covsyn}

{In this subsection we prove the convergence of Algorithm~\ref{alg:syn}. We first re-derive the iteration \eqref{alg:cv} using operator splitting techniques and identify it as an averaged operator. Next, we apply {Theorem \ref{thm:cvg}} and establish its convergence property.}

Let $\rho_{\min}:=\lambda_{\min}(G)>0$ be the smallest eigenvalue of $G$.
{In the following assumption, we specify the properties of the cost functions $s_i$ and $r_i$, and the step-size $\alpha$.}
\begin{assumption}\label{assumfunc}
	$ $
	\begin{enumerate}
		\item The functions $s_i$ and $r_i$ are closed, proper and convex.
		\item The functions $s_i$ are differentiable and satisfy:
		$$\|\nabla s_i(x)-\nabla s_i(\widetilde{x})\|\leq L_i\|x-\widetilde{x}\|,\quad\forall x,\widetilde{x}\in\RR^p,$$
		where $L_i>0$ is the Lipschitz constant.
		\item The parameter $\alpha$ in synchronous algorithm \eqref{syn-decentralized} and asynchronous  algorithm \eqref{update_asyn} satisfies $0<\alpha<{2\rho_{\min}}/{L}$, where $L:=\max_i L_i$.
	\end{enumerate}
\end{assumption}


Since $\vr$ and $\vs$ are convex, the solution to the saddle point problem~\eqref{prob:saddle} is the same as the solution to the following Karush-Kuhn-Tucker (KKT) system:
\eq{\label{KKT}
0\in \bigg(\hspace{-1.5mm}\underbrace{\begin{bmatrix}
                \nabla\vs & 0\\
                0 & 0
        \end{bmatrix}}_{\mbox{operator}~\cA}\hspace{-1.5mm} + \hspace{-1 mm} \underbrace{
        \begin{bmatrix}
                \partial\vr & 0 \\
                0 & 0
        \end{bmatrix}+\begin{bmatrix}
        0&\frac{1}{\alpha}V^\top\\
        -\frac{1}{\alpha}V&0
\end{bmatrix}}_{\mbox{operator}~\cB}\hspace{-1 mm} \bigg) \underbrace{\begin{bmatrix}
X\\
Y
\end{bmatrix}}_{Z},
}
which can be written as 
$$0\in\cA Z+\cB Z,\quad \text{with}\quad  Z:=\begin{bmatrix}X\\ Y\end{bmatrix}\in \RR^{(n+m)\times p}.$$
For any positive definite matrix $M\in \RR^{(n+m)\times (n+m)}$, we have
\begin{align}\label{xcmn8}
0\in\cA Z+\cB Z \Leftrightarrow&~MZ-\cA Z \in MZ+\cB Z\\
\Leftrightarrow&~Z-M^{-1}\cA Z \in Z+M^{-1}\cB Z\\
\Leftrightarrow&~Z=(\cI+M^{-1}\cB)^{-1}(\cI-M^{-1}\cA)Z,
\end{align}
Note that $(\cI+M^{-1}\cB)^{-1}$ is well-defined and single-valued; it is called \emph{resolvent} \cite{bauschke2011convex}.
From now on, we let 
\eq{
M:=\frac{1}{\alpha}G \succ 0
\label{M=G/alpha}
}
and define the operator
\begin{equation}
\cT:=(\cI+M^{-1}\cB)^{-1}(\cI-M^{-1}\cA).\label{operator}\end{equation}
From \eqref{KKT}, \eqref{xcmn8} and the definition of operator $\cT$ in \eqref{operator}, we conclude that the solutions to the KKT system~\eqref{KKT} coincide with the fixed points of $\cT$.

Now we consider the fixed-point iteration 
\eq{
Z^{k+1}=\cT Z^k=(\cI+M^{-1}\cB)^{-1}(\cI-M^{-1}\cA)Z^k,\label{908098asdkjh}
}
which reduces to 
\eq{\label{23089a}
MZ^k-\cA Z^k \in MZ^{k+1}+{\cB}Z^{k+1}.}
{Substituting $M$ defined in \eqref{M=G/alpha} into \eqref{23089a} and multiplying $\alpha$ to both sides, we will achieve 
}
\begin{align*}
\begin{cases}
        \begin{array}{rl}
                \hspace{-2mm}X^k \hspace{-1mm}+\hspace{-1mm} V^\top\hspace{-0.5mm} Y^{k}\hspace{-1mm}-\hspace{-1mm}\alpha\hspace{-0.5mm} \nabla \vs(X^k) \hspace{-3mm}&\in \hspace{-0.5mm} X^{k+1} \hspace{-1mm}+\hspace{-1mm} 2V^\top\hspace{-0.5mm} Y^{k\hspace{-0.2mm}+\hspace{-0.2mm}1} \hspace{-0.5mm}+\hspace{-0.5mm} \alpha\partial \vr (X^{k\hspace{-0.2mm}+\hspace{-0.2mm}1}),\\
                Y^k + VX^{k}&= Y^{k+1}+ VX^{k+1} -VX^{k+1},
        \end{array}
\end{cases}
\end{align*}
%
which, by cancellation, is further equivalent to Recursion~\eqref{alg:cv}. 

{Until now, we have rewritten the primal-dual algorithm \eqref{alg:cv} as a fixed-point iteration \eqref{908098asdkjh} with operator $\cT$ defined in \eqref{operator}. Besides, we also established that $\mathbf{Fix} \cT$ coincide with the solutions to the KKT system \eqref{KKT}. What still needs to be proved is that $\{Z^k\}$ generated through the fixed-point iteration \eqref{908098asdkjh} will converge to $\mathbf{Fix} \cT$. }
Below is a classic important result of the operator $\cT$. For a proof, see~\cite{davis2015convergence}.
\begin{theorem}\label{thm:average}
Under Assumption~\ref{assumfunc} and the above-defined norm with $M = G/\alpha$, there exists a nonexpansive operator $\cO$ and some $\beta \in (0,1)$ such that the operator $\cT$ defined in \eqref{operator} satisfies $\cT=(1-\beta)I+\beta\cO$. Hence, $\cT$ is an averaged operator. 
\end{theorem}
{The convergence of the synchronous update~\eqref{alg:cv} follows directly from Theorems~\ref{thm:cvg} and \ref{thm:average}}.
\begin{corollary}\label{synconv}
Under Assumption~\ref{assumfunc}, Algorithm~\ref{alg:syn} produces $Z^k$ that converges to a solution $Z^*=[X^*;Y^*]$ to the KKT system~\eqref{KKT}, which is also a saddle point to Problem~\eqref{prob:saddle}. Therefore, $X^*$ is a solution to Problem~\eqref{prob-compact-V}.
\end{corollary}
Corollary \ref{synconv} states that if we run algorithm~\ref{alg:syn} in the synchronous fashion, the agents' local variables will converge to a consensual solution to the problem~\eqref{prob-general}.


%% file: analysis.tex
\subsection{{Convergence of Algorithm~\ref{alg:asyn}}}
{Algorithm~\ref{alg:asyn} is an asynchronous version that involves random variables. In this subsection we will} establish its almost sure (a.s.) convergence (i.e., with probability 1). 
The analysis is inspired by~\cite{Peng_2015_AROCK} and~\cite{peng2016coordinate}.

{\subsubsection{Algorithm reformulation}It was shown in Sec. \ref{sec:covsyn} that Algorithm~\ref{alg:syn} can be rewritten as a fixed-point iteration \eqref{908098asdkjh}. Now we let $\cT_i$ be the operator corresponding to the update performed by agent $i$, i.e., $\cT_i$ only updates $x^i$ and $\{y^e\}_{e\in\cL_i}$ (the definition of $\cL_i$ can be referred to \eqref{Li}) and leaves the other variables unchanged.
{We define
	\eq{
	\cI_i = \{i\} \cup \{n+e|\mbox{ $e$ is the edge index of $(i,j)\in \cL_i$ }\}
	}
to be the set of all row indices of $Z$ that agent $i$ needs to update. The operator $\cT_i: \RR^{(m+n)\times p}\to \RR^{(m+n)\times p}$ is defined as follows. For any $Z\in \RR^{(m+n)\times p}$, $\cT_i Z \in \RR^{(m+n)\times p}$ is a matrix with
\eq{\label{Ti}
(\cT_i Z)_j :=
\begin{cases}
	\begin{array}{cl}
		(\cT Z)_j, & \mbox{if }j\in \cI_i;\\
		Z_j, & \mbox{if } j\notin \cI_i,
	\end{array}
\end{cases}
}
where $Z_j$, $(\cT Z)_j$ and $(\cT_i Z)_j$ are the $j$-th row of $Z$, $\cT Z$ and $\cT_i Z$, respectively (see the notation section \ref{sec:notation}).
	Moreover, we define
	\eq{\label{S=I-T}
		\cS:=I-\cT, \quad\quad \cS_i:= I - \cT_i.
		}
	{Based on the definition of the operator $\cT_i$, for any $Z\in \RR^{(m+n)\times p}$, $\cS_i Z \in \RR^{(m+n)\times p}$ is also a matrix with
		\eq{\label{Si}
			(\cS_i Z)_j :=
			\begin{cases}
				\begin{array}{cl}
					(Z - \cT_iZ)_j,& \mbox{if $j\in \cI_i$;}\\
					0, & \mbox{if $j\notin \cI_i$.}
				\end{array}
			\end{cases}
		}
	}
	We also define the delayed variable used in iteration $k$ as
	{\color{black}
		\eq{\widehat{Z}^k:=
			\ba{c}
			X^{k-\tau^k}\\Y^{k-\delta^k}
			\ea \in \RR^{(m+n)\times p}.\label{zhat}}

		Suppose agent $i_k$ is activated at iteration $k$, with definitions \eqref{Ti} and \eqref{zhat} it can be verified that recursions \eqref{asyn-compact-1} and \eqref{asyn-compact-2} are equivalent to:
		\eq{
			\begin{cases}
				\begin{array}{rl}
					\widetilde{Z}^{k} & = \cT_{i_k} \widehat{Z}^k, \\
					Z^{k+1} & =  Z^k - \eta_{i_k} \big( \widehat{Z}^k - \widetilde{Z}^{k} \big).
				\end{array}
			\end{cases} \label{c972}
		}
		From recursion \eqref{c972} we further have
		\eq{
			Z^{k+1}  = Z^k - \eta_{i_k} \big( \widehat{Z}^k - \cT_{i_k} \widehat{Z}^k \big)\overset{\eqref{Si}}{=}Z^k - \eta_{i_k} \cS_{i_k}\widehat{Z}^k. \label{i237}
		}
	}Therefore, the asynchronous update algorithm \eqref{update_asyn} can be viewed as a stochastic coordinate-update iteration based on delayed variables~\cite{Peng_2015_AROCK,peng2016coordinate}, with each coordinate update taking the form of~\eqref{i237}.
	
	\vspace{1mm}
	\subsubsection{Relation between $\widehat{Z}^k$ and ${Z}^k$}
	{The following assumption introduces a uniform upper bound for the random delays.}
	\begin{assumption}\label{ass-delay}
		{At any iteration $k$,}
		the delays $\tau_j^k,j=1,2,\ldots,n$ and $\delta_e^k,e=1,2,\ldots,m$ defined in~\eqref{update_asyn} have a uniform upper bound $\tau>0$.
	\end{assumption}
	As we are finishing this paper, the recent work \cite{Hannah2016unbounded} has relaxed the above assumption by associating step sizes to potentially unbounded delays. {We still keep Assumption~\ref{ass-delay} in our main analysis and shall discuss the case of unbounded delays in Section~\ref{sec:ubd}} 
	
	Under Assumption \ref{ass-delay}, the relation between $\widehat{Z}^k$ and $Z^k$ can be characterized in the following lemma \cite{Peng_2015_AROCK}:
	\begin{lemma}\label{lemma-hat-z-relation}
		Under Assumption \ref{ass-delay}, it holds that
		\begin{align}
		\widehat{Z}^k=Z^k+\sum_{d\in J(k)}(Z^d-Z^{d+1}),\label{eqn:inconsist}
		\end{align}
		where $J(k)\subseteq \{k-1,...,k-\tau\}$ is an index set.
	\end{lemma}
}
\noindent {The proof of relation \eqref{eqn:inconsist} can be referred to \cite{Peng_2015_AROCK}.
}

\vspace{1mm}
\subsubsection{Convergence analysis} Now we introduce an assumption about the activation probability of each agent.

\begin{assumption}\label{assum:random}
	For any $k>0$, let  $i_k$ be the index of the agent that is responsible for the $k$-th completed update. It is assumed that each $i_k$ is a random variable. The random variable $i_k$ is independent of $i_1,i_2,\cdots,i_{k-1}$ as $$P(i_k=i)=:q_i>0,$$
	where $q_i$'s are constants.
\end{assumption}

This assumption is satisfied under either of the following scenarios: (i) every agent $i$ is activated following an independent Poisson process with parameter $\lambda_i$, and any computation occurring at  agent $i$ is instant, leading to $q_i = \lambda_i / (\sum_{i=1}^n \lambda_i)$; (ii) every agent $i$ runs continuously, and the duration of each round of computation follows the exponential distribution $\exp(\beta_i)$, leading to $q_i=\beta_i^{-1} / (\sum_{i=1}^n \beta_i^{-1})$. {Scenarios (i) and (ii) often appear as assumptions in the existing literature.}

\begin{definition}
	Let $(\Omega,\cF,P)$ be the probability space we work with, where $\Omega$, $\cF$ and $P$ are the sample space, $\sigma$-algebra and the probability measure, respectively. Define
	$$\cZ^k:=\sigma (Z^0,\widehat{Z}^0,Z^1,\widehat Z^1,\cdots,Z^k,\widehat Z^k)$$
	 to be the $\sigma$-algebra generated by $Z^0,\widehat Z^0,\cdots,Z^k,\widehat Z^k$.
\end{definition}
\begin{assumption}\label{ass-4}
	Throughout our analysis, we assume
	\begin{equation}
	P(i_k=i|\cZ^k)=P(i_k=i)=q_i,\quad\forall i,k.
	\end{equation}
\end{assumption}
{\begin{remark}
	Assumption~\ref{ass-4} 
	indicates that the index responsible for the $k$-th completed update, $i_k$, is independent of the delays in different rows of $\widehat Z^k$.  Although not always practical, it is a key assumption for our proof to go through. One of the cases for this assumption to hold is when every row of $\widehat Z^k$ always has the maximum delay $\tau$. In reality, what happens is between this worst case and the no-delay case. Besides, Assumption \ref{ass-4} is also a common assumption in the recent literature of stochastic asynchronous algorithms; see \cite{Peng_2015_AROCK} and the references therein.
	
\end{remark}}

{Now we are ready to prove the convergence of Algorithm~\ref{alg:asyn}. For simplicity, we define }
\begin{align}
\bar{Z}^{k+1}&:=Z^k-\eta\, \cS\, \widehat{Z}^k,\quad q_{\min}:=\min_i \{q_i\}>0,\label{zbar}\\
\lambda_{\max}&:=\lambda_{\max}(M),\quad \hspace{4.5mm}\lambda_{\min}:=\lambda_{\min}(M),\nonumber
\end{align}
{where $\eta>0$ is some constant, and $M\succ 0$ is defined in \eqref{M=G/alpha}.}
{Note that $\bar{Z}^{k+1}$ is the full update as if all agents synchronously compute at the $k$-th iteration. It is used only for analysis.}

\begin{lemma}\label{lemma:a-avg}Define $\cQ := I - \cP$. When $\cP:\RR^{(m+n)\times p}\to \RR^{(m+n)\times p}$ is nonexpansive under the norm $\|\cdot\|_A$, we have
\begin{equation}\label{eqn:alpha_avg}
\langle Z-\widetilde{Z},\cQ Z- \cQ\,\widetilde{Z}\rangle_A \ge \frac{1}{2}\|\cQ Z- \cQ\,\widetilde{Z}\|_A^2,
\end{equation}
where $Z,\widetilde{Z}\in \RR^{(m+n)\times p}$ are two arbitrary matrices.
\end{lemma}
\begin{proof} By nonexpansiveness of $\cP$, we have
\eq{\label{xlhadu}
\|\cP Z-\cP\,\widetilde Z\|_A^2 \le \| Z-\widetilde Z\|_A^2, \ \ \forall\  Z, \widetilde{ Z} \in \RR^{(m+n)\times p}.
}	
Notice that
\begin{align}
\label{aslkhfa}
&\ \|\cP Z-\cP\,\widetilde Z\|_A^2 = \| Z-\widetilde Z-(\cQ Z-\cQ\,\widetilde Z)\|^2_A\nnb
=&\ {\| Z\hspace{-0.8mm}-\hspace{-0.8mm}\widetilde Z\|^2_A \hspace{-0.8mm}-\hspace{-0.8mm} 2\langle  Z\hspace{-0.8mm}-\hspace{-0.8mm}\widetilde Z,\cQ Z\hspace{-0.8mm}-\hspace{-0.8mm}\cQ\,\widetilde Z\rangle_A \hspace{-0.8mm}+\hspace{-0.8mm} \|\cQ Z\hspace{-0.8mm}-\hspace{-0.8mm}\cQ\,\widetilde Z\|_A^2}.
\end{align}
Substituting \eqref{aslkhfa} into \eqref{xlhadu}, we achieve \eqref{eqn:alpha_avg}.
\end{proof}

Let $\kappa$ be its condition number of matrix $G$. Since $M=G/\alpha$, $\kappa$ is also the condition number of $M$. The following lemma establishes the relation between $\cS_i \widehat{ Z}$ and $\cS \widehat{ Z}$.
\begin{lemma}
Recall the definition of $\cS, \cS_i$ and $M$ in \eqref{S=I-T} and \eqref{M=G/alpha}. It holds that
\begin{align}
\sum_{i=1}^n\cS_i\widehat{ Z}^k=\cS\,\widehat{ Z}^k, \quad \sum_{i=1}^n\|\cS_i\widehat{ Z}^k\|_M^2\leq \kappa\|\cS\,\widehat{ Z}^k\|_M^2.\label{eqn:bound_lemma1}
\end{align}
\end{lemma}
\begin{proof} The first part comes immediately from the definition of $\cS$ and $\cS_i$ in \eqref{S=I-T} and \eqref{i237}. For the second part,
\begin{align}
\sum_{i=1}^n\|\cS_i\widehat{ Z}^k\|_M^2&\leq \sum_{i=1}^n \lambda_{\max}\|\cS_i\widehat{ Z}^k\|^2_{\rm F} = \lambda_{\max}\|\cS\widehat{ Z}^k\|^2_{\rm F}\\
&\leq {\lambda_{\max}\over\lambda_{\min}}\|\cS\widehat{ Z}^k\|^2_M=\kappa\|\cS\widehat{ Z}^k\|^2_M.
\end{align}
\end{proof}
The next lemma shows that the conditional expectation of the distance between $ Z^{k+1}$
and any $ Z^*\in \mathbf{Fix} \cT$ for given $\cZ^k$ has an
upper bound that depends on $\cZ^k$ and $ Z^*$ only. {From now on, the relaxation parameters will be set as
		\begin{equation}\label{eta_i_def}
		\eta_i=\frac{\eta}{nq_i},
		\end{equation}
		where $\eta>0$ is the constant appearing at \eqref{zbar}.}
\begin{lemma}\label{lemma:fund}
Let $\{Z^k\} _{k\geq 0}$ be the sequence generated by Algorithm~
\ref{alg:asyn}.  Then for
any $ Z^*\in \mathbf{Fix} \cT$, we have
\begin{align}
& \mathbb{E}\big(\| Z^{k+1} -  Z^* \|_M^2 \,\big|\, \cZ^k\big) \nnb
  \leq\ &  \| Z^{k} -  Z^*\|_M^2  +{\xi\over n}\sum_{k-\tau\leq d<k}\| Z^d- Z^{d+1}\|_M^2\nonumber\\
&\quad + {1\over n}\left({{\tau}\over \xi}+{\kappa\over nq_{\min}}-{1\over \eta}\right)\| Z^k-\bar  Z^{k+1}\|_M^2,\label{eqn:fund_inquality0}
\end{align}%
where $\mathbb{E}(\cdot\,|\,\cZ^k)$ denotes conditional expectation  on $\cZ^k$ and $\xi$ is an arbitrary positive number.
%
\end{lemma}
\begin{proof}
We have
\eq{
&\ \mathbb{E}\left(\| Z^{k+1} -  Z^*\|_M^2\,|\,\cZ^k\right)\\
=&\ \mathbb{E}\Big( \Big\| Z^{k}  -
\frac{\eta}{nq_{i_k}}\cS_{i_k} \widehat{ Z}^{k}-  Z^*
\Big\|_M^2\,\Big|\,\cZ^k\Big)\quad\mbox{(by}~\eqref{i237},\eqref{eta_i_def}\mbox{)} \nonumber
}
\eq{
=&\ \mathbb{E}\Big(  \frac{2\eta}{nq_{i_k}} \big\langle
\cS_{i_k} \widehat{ Z}^{k},  Z^* -  Z^k \big\rangle_M +
\frac{\eta^2}{n^2q_{i_k}^2}
\|\cS_{i_k}\widehat{ Z}^{k}\|_M^2\,\big|\,\cZ^k\Big) \nnb
&\quad +\| Z^k - Z^*\|_M^2 \nnb
\overset{(a)}{=}&\ \frac{2\eta}{n} \sum_{i=1}^n\left\langle \cS_i
\widehat{ Z}^{k},  Z^* -  Z^k \right\rangle_M +
\frac{\eta^2}{n^2}\sum_{i=1}^n\frac{1}{q_i}\|\cS_i\widehat{ Z}^{k}\|_M^2\\
&\quad + \| Z^k -
 Z^*\|_M^2 \nnb
\overset{\eqref{eqn:bound_lemma1}}{=}&\ \| Z^k \hspace{-0.8mm}
- \hspace{-0.8mm}  Z^*\|_M^2 \hspace{-0.5mm} +  \hspace{-0.5mm} \frac{2\eta}{n} \left\langle \hspace{-0.5mm} \cS \widehat{ Z}^{k},  Z^* \hspace{-0.8mm}
- \hspace{-0.8mm}  Z^k \hspace{-0.5mm}
\right\rangle_M  \hspace{-0.5mm} +  \hspace{-0.5mm} \frac{\eta^2}{n^2} \sum_{i=1}^n\frac{1}{q_i}
\|\cS_i\widehat{ Z}^{k}\|_M^2,\label{eqn:equality_inconsistent}
}
where the equality (a) holds because of Assumptions \ref{assum:random} and \ref{ass-4}. On the other hand, note that
\eq{\label{term2}
\sum_{i=1}^n\frac{1}{q_i} \|\cS_i\widehat{ Z}^{k}\|_M^2&\le\frac{1}{q_{\min}}
\sum_{i=1}^n\|\cS_i\widehat{ Z}^{k}\|_M^2\overset{\eqref{eqn:bound_lemma1}}{\leq}
\frac{\kappa}{q_{\min}} \sum_{i=1}^n\|\cS\widehat{ Z}^{k}\|^2_M \nnb
&\overset{\eqref{zbar}}{=}\frac{\kappa}{\eta^2q_{\min}}\| Z^k-\bar{ Z}^{k+1}\|_M^2,
}
and
\eq{
&\left\langle \cS \widehat{ Z}^{k},  Z^* -  Z^k \right\rangle_M\nnb
\overset{\eqref{eqn:inconsist}}{=}&\langle \cS \widehat{ Z}^{k}, Z^* - \widehat{ Z}^k + \sum_{d\in J(k)} ( Z^{d} -
 Z^{d+1})\rangle_M \nnb
\overset{\eqref{zbar}}=&\langle \cS
\widehat{ Z}^{k},  Z^* - \widehat{ Z}^k\rangle_M + \frac{1}{\eta}\sum_{d\in J(k)}\langle
 Z^k-\bar{ Z}^{k+1},  Z^{d} -  Z^{d+1}\rangle_M \nnb
\overset{(b)}{\le}& \langle \cS
\widehat{ Z}^{k}-\cS  Z^*,  Z^* - \widehat{ Z}^k\rangle_M \nnb
&+\frac{1}{2\eta}\sum_{d\in
J(k)}\left(\frac{1}{\xi}\| Z^k-\bar{ Z}^{k+1}\|_M^2+ \xi\| Z^{d} -
 Z^{d+1}\|_M^2\right) \nnb
\overset{\eqref{eqn:alpha_avg}}\le&
\hspace{-1.2mm}-\hspace{-1mm}\frac{1}{2}\hspace{-0.2mm}\|\hspace{-0.2mm}\cS \widehat{ Z}^{k}\hspace{-0.2mm}\|_M^2\hspace{-0.8mm}+\hspace{-0.8mm}\frac{1}{2\eta}\hspace{-1mm}\sum_{d\in
J(k)}\hspace{-2mm}\left(\hspace{-0.5mm}\frac{1}{\xi}\| Z^k\hspace{-0.8mm}-\hspace{-0.8mm}\bar{ Z}^{k+1}\|_M^2 \hspace{-0.8mm}+\hspace{-0.8mm} \xi\| Z^{d} \hspace{-0.8mm}-\hspace{-0.8mm}
 Z^{d+1}\|_M^2\hspace{-1.5mm}\right) \nnb
\overset{\eqref{zbar}}=&-\frac{1}{2\eta^2}\| Z^k-\bar{ Z}^{k+1}\|_M^2+
\frac{|J(k)|}{2\xi\eta}\| Z^k-\bar{ Z}^{k+1}\|_M^2 \nnb
&+\frac{\xi}{2\eta}\sum_{d\in
J(k)}\| Z^{d} -  Z^{d+1}\|_M^2, \label{term1}
}
where the inequality (b) follows from Young's inequality and the fact that $\cS  Z^* = (I - \cT) Z^*=0$ since $ Z^*$ is a fixed point of $\cT$. Substituting~\eqref{term2} and~\eqref{term1} into~\eqref{eqn:equality_inconsistent} and using $J(k)\subseteq \{k-1,...,k-\tau\}$ we get the desired result.\hfill\end{proof}
The term ${\xi\over n}\sum_{k-\tau\leq d<k}\|Z^d-Z^{d+1}\|^2$ in the inequality \eqref{eqn:fund_inquality0} appears because of the delay.
Next we stack $\tau+1$ iterates together to form a new vector and introduce a new metric in order to absorb these terms.

Let $\cH^{\tau+1}=\prod_{i=0}^{\tau}\cH$ be a product space (Recall $\cH$ denotes $\RR^{(m+n)\times p}$, see Definition \ref{defi:4}). For any $(Z_0,\ldots,Z_{\tau}), (\widetilde{Z}_0,\ldots,\widetilde{Z}_{\tau}) \in \cH^{\tau+1}$, we let $\langle\cdot,\cdot \rangle$ be the induced inner product, i.e.,
$ \langle (Z_0,\ldots,Z_{\tau}),(\widetilde{ Z}_0,\ldots,\widetilde{ Z}_{\tau})\rangle  :=\ \sum_{i=0}^{\tau}\langle  Z_i,\widetilde{ Z}_i\rangle_M=\sum_{i=0}^{\tau}\tr(Z_i^\top M \widetilde{Z}_i).
$
Define a matrix $U'\in \RR^{(\tau+1)\times(\tau+1)}$ as
\eq{U' := U_1 + U_2,}
where
\eq{
U_1:=\begin{bmatrix}1 & 0 & \cdots &0\\
	0 & 0 &\cdots & 0\\ \vdots &\vdots & \ddots & \vdots\\
	0 & 0 &\cdots & 0 \end{bmatrix}
}
\eq{
U_2:=&\
\sqrt{\frac{q_{\min}}{\kappa}}\begin{bmatrix} \tau & -\tau &  & \\
-\tau & 2\tau-1 & 1-\tau & \\
 & 1-\tau & 2\tau-3 & 2-\tau  & \\
 & & \ddots & \ddots & \ddots &\\
 & & & -2 & 3  & -1 \\
 & & & &-1 & 1
\end{bmatrix},
}
and let $U=U'\otimes I$ where $\otimes$ represents the Kronecker product and $I$ is the identity operator.
For a given $(A_0,\cdots,A_\tau)\in\cH^{\tau+1}$,
$$(B_0,\cdots,B_\tau)=U(A_0,\cdots,A_\tau)$$
is given by:
\eq{
 B_0&=
 A_0+\tau\sqrt{\frac{q_{\min}}{\kappa}} ( A_0- A_1), \nnb
 B_i &=
\sqrt{\frac{q_{\min}}{\kappa}}[(i\hspace{-0.8mm}-\hspace{-0.8mm} \tau\hspace{-0.8mm}-\hspace{-0.8mm}1) A_{i-1} \hspace{-0.8mm}+\hspace{-0.8mm} (2\tau \hspace{-0.8mm}-\hspace{-0.8mm} 2i \hspace{-0.8mm}+\hspace{-0.8mm} 1) A_i+(i-\tau) A_{i+1}],\nnb
 B_{\tau}&=\sqrt{\frac{q_{\min}}{\kappa}} ( A_{\tau}- A_{\tau-1}),\label{a-b-relation}
}
where the index $i$ for $B_i$ is from $1$ to $\tau+1$.The linear operator $U$ is a self-adjoint and positive definite since $U'$ is
	symmetric and positive definite. We define $\langle\cdot,
	\cdot\rangle_U=\langle\cdot,U\cdot\rangle$ as the $U$-weighted inner
	product and $\|\cdot\|_U$ as the induced norm.
We further let
\begin{align*}
\vZ^k&:=( Z^k, Z^{k-1},\ldots, Z^{k-\tau})\in \cH^{\tau+1},~k\ge 0,\\
\vZ^*&:=( Z^*,\ldots, Z^*)\in\cH^{\tau+1},
\end{align*}
where $ Z^{k}= Z^{0}$ for $k<0$. {We have
\begin{align}
&\hspace{-2mm}\|\vZ^k-\vZ^*\|_U^2 \nnb
\hspace{-4mm}\overset{\eqref{a-b-relation}}{=}&\| Z^{k} \hspace{-1.5mm}-\hspace{-1.5mm}  Z^*\|_M^2 \hspace{-1.1mm}+\hspace{-1.2mm}\sqrt{q_{\min}\over \kappa}\hspace{-1.5mm}\sum_{d=k-\tau}^{k-1}\hspace{-2mm} (d\hspace{-1mm}-\hspace{-1mm}(k\hspace{-1mm}-\hspace{-1mm}\tau)\hspace{-1mm}+\hspace{-1mm}1) \|
 Z^{d} \hspace{-1mm}-\hspace{-1mm}  Z^{d+1}\|_M^2, \label{eqn:xi}
\end{align}
and the following fundamental inequality.
\begin{theorem}[Fundamental inequality]\label{thm:fund_inquality}
Let $\{Z^k\}_{k\geq 0}$ be the sequence generated by Algorithm~\ref{alg:asyn}. Then for any $\vZ^*=( Z^*,\ldots, Z^*)$, it holds that 
\begin{align}
&\hspace{-2mm} \mathbb{E}\big(\|\vZ^{k+1}-\vZ^*\|^2_U \big|\, \cZ^k\big) \nnb
\hspace{-3mm}\leq&\  \|\vZ^k\hspace{-1mm}-\hspace{-1mm}\vZ^*\|^2_U \hspace{-1mm} - \hspace{-1mm} \frac{1}{n}
\hspace{-1mm} \left( \frac{1}{ \eta} \hspace{-1mm}-\hspace{-1mm} \frac{2\tau\sqrt{\kappa}}{n\sqrt{q_{\min}}} \hspace{-1mm}-\hspace{-1mm}
{\kappa\over nq_{\min}} \hspace{-1mm}\right)\hspace{-1mm}
\|\bar{ Z}^{k+1} \hspace{-1mm}-\hspace{-1mm}  Z^k \|_M^2.\label{eqn:fund_inquality}
\end{align}
\end{theorem}}
\begin{proof}
Let $\xi=n\sqrt{{q_{\min}}/{\kappa}}$. We have
\eq{
&\ \mathbb{E} (\|\vZ^{k+1}-\vZ^*\|^2_U | \cZ^k)  \nnb
\overset{\eqref{eqn:xi}}= &\  \mathbb{E} (\hspace{-0.3mm}\|\hspace{-0.5mm} Z^{k\hspace{-0.3mm}+\hspace{-0.3mm}1} \hspace{-0.8mm}-\hspace{-0.8mm}  Z^*\hspace{-0.5mm}\|_M^2| \cZ^k \hspace{-0.3mm}) \nnb
&\quad + \xi\textstyle{\sum_{d=k\hspace{-0.2mm}+\hspace{-0.2mm}1\hspace{-0.2mm}-\hspace{-0.2mm}\tau}^{k} \hspace{-4mm} \frac{d\hspace{-0.2mm}-\hspace{-0.2mm}(k\hspace{-0.2mm}-\hspace{-0.2mm}\tau)}{n}} \mathbb{E} (\hspace{-0.5mm}\|\hspace{-0.5mm}  Z^{ d} \hspace{-0.8mm}-\hspace{-0.8mm}
 Z^{d\hspace{-0.3mm}+\hspace{-0.3mm}1}\hspace{-0.5mm}\|_M^2 | \cZ^k\hspace{-0.5mm}) \nnb
  \overset{\eqref{i237}}= &\ \mathbb{E} (\| Z^{k+1} -
   Z^*\|_M^2| \cZ^k) + \textstyle{\frac{\xi\tau}{n}}
  \mathbb{E}(\textstyle{\frac{\eta^2}{n^2q_{i_k}^2}}\| S_{i_k}\widehat  Z^k\|_M^2|\cZ^k) \nnb
  &\quad +   \xi\textstyle{\sum_{d=k+1-\tau}^{k-1}} \frac{d\hspace{-0.2mm}-\hspace{-0.2mm}(k\hspace{-0.2mm}-\hspace{-0.2mm}\tau)}{n} \|  Z^{d} -  Z^{d+1}\|_M^2 \nnb
          \overset{\eqref{term2}}{\le} &\ \mathbb{E} (\| Z^{k+1} -  Z^*\|_M^2| \cZ^k) +
          \textstyle{\frac{\xi\tau\kappa}{n^3q_{\min}}} \|  Z^{k } -
          \bar{ Z}^{k+1}\|_M^2 \nnb
          &\quad +   \xi\textstyle{\sum_{d=k+1-\tau}^{k-1}} \frac{d\hspace{-0.2mm}-\hspace{-0.2mm}(k\hspace{-0.2mm}-\hspace{-0.2mm}\tau)}{n} \|  Z^{d} -  Z^{d+1}\|_M^2 \nnb
  \overset{\eqref{eqn:fund_inquality0}}\leq &\ \| Z^k \hspace{-1mm}-\hspace{-0.8mm}  Z^*\|_M^2 \hspace{-1mm}+\hspace{-1mm}
  \textstyle{\frac{1}{n}}\hspace{-1mm} \left(\hspace{-1mm}\textstyle{\tau\over \xi} \hspace{-1mm}+\hspace{-1mm}
  \textstyle{\frac{\xi\tau\kappa}{n^2q_{\min}}} \hspace{-1mm}+\hspace{-1mm}
  {\kappa\over nq_{\min}} \hspace{-1mm}-\hspace{-0.8mm} \frac{1}{\eta}\right) \| Z^k \hspace{-1mm}-\hspace{-0.8mm}
  \bar{ Z}^{k+1}\|_M^2   \nnb
&\textstyle{+\frac{\xi}{n}\sum_{k-\tau\leq d<k}\|{ Z}^{d} - { Z}^{d+1}\|_M^2}   \nnb
& \textstyle{+
\xi\sum_{d=k+1-\tau}^{k-1} \frac{d-(k-\tau)}{n} \|  Z^{d} -  Z^{d+1}\|_M^2} \nnb
=&\ \| Z^k \hspace{-1mm}-\hspace{-0.8mm}  Z^*\|_M^2 \hspace{-1mm}+\hspace{-1mm}
\textstyle{\frac{1}{n}}\hspace{-1mm} \left(\hspace{-1mm}\textstyle{\frac{\tau}{\xi}} \hspace{-1mm}+\hspace{-1mm}
\textstyle{\frac{\xi\tau\kappa}{n^2q_{\min}}} \hspace{-1mm}+\hspace{-1mm}
{\kappa\over nq_{\min}} \hspace{-1mm}-\hspace{-0.8mm} \frac{1}{\eta}\right) \| Z^k \hspace{-1mm}-\hspace{-0.8mm}
\bar{ Z}^{k+1}\|_M^2   \nnb
&\ +
\textstyle{\frac{\xi}{n}\sum_{d=k-\tau}^{k-1} {(d-(k-\tau)+1)}} \|  Z^{d} -  Z^{d+1}\|_M^2\nnb
 \overset{\eqref{eqn:xi}}= &\ \|\vZ^k-\vZ^*\|^2_U  \hspace{-0.8mm}+\hspace{-0.8mm} \textstyle{\frac{1}{n}}
 \Big(\textstyle{\frac{2\tau\sqrt{\kappa}}{n\sqrt{q_{\min}}}} \hspace{-0.8mm}+\hspace{-0.8mm}
 {\kappa\over nq_{\min}} \hspace{-0.8mm}-\hspace{-0.8mm} \frac{1}{ \eta}\Big) \| Z^k \hspace{-0.8mm}-\hspace{-0.8mm}
 \bar{ Z}^{k+1}\|_M^2.
}
Hence, the desired inequality \eqref{eqn:fund_inquality} holds.
\hfill
\end{proof}
\begin{remark}[Stochastic Fej{\'e}r monotonicity~\cite{combettes2015stochastic}]
From~\eqref{eqn:fund_inquality}, suppose
\eq{\label{28abncmko}
0<\eta<\frac{nq_{\min}}{2\tau
\sqrt{\kappa q_{\min}}+\kappa},
}
then we have
$$\EE(\|\vZ^{k+1}-\vZ^*\|_U^2|\cZ^k)\leq\|\vZ^k-\vZ^*\|_U^2,$$
i.e., the sequence $\{\vZ^k\}_{k\geq0}$ is stochastically Fej{\'e}r monotone, which means the covariance of $\vZ^k$ is non-increasing as the iteation $k$ evolves.
\end{remark}
{Based on Theorem~\ref{thm:fund_inquality}, we have the following corollary:
\begin{corollary}\label{cor:summable}
When $\eta$ satisfies \eqref{28abncmko}, we have
\begin{enumerate}
\item $\sum_{k=0}^{\infty}\| Z^k - \bar{ Z}^{k+1}\|_M^2 < \infty$ a.s.;\\
\item the sequence $\{\|\vZ^k-\vZ^*\|_U^2\}_{k\geq0}$ converges to a $[0, +\infty)$-valued random variable a.s.
\end{enumerate}
\end{corollary}
\begin{proof}
The condition \eqref{28abncmko} implies that $\frac{1}{ \eta}-\frac{2\tau\sqrt{\kappa}}{n\sqrt{q_{\min}}} -
{\kappa\over nq_{\min}}>0.$
Applying \cite[Theorem~1]{robbins1985convergence} to \ref{eqn:fund_inquality} directly gives the two conclusions.
\end{proof}}

The following lemma establishes the convergence properties of the sequences $\{\vZ^k\}_{k=1}^\infty$ and $\{ Z^k\}_{k=1}^\infty$.
\begin{lemma}\label{lemma:convergence}
Define $\vS^*=\{( Z^*,\ldots, Z^*)| Z^*\in\mathbf{Fix} \cT\}$, let $( Z^k)_{k\geq0}\subset \cH$ be the sequence generated by Algorithm~\ref{alg:asyn}, $\eta
\in (0, \eta_{\max}]$ for certain $\eta_{\max}$ satisfying \eqref{28abncmko}. Let $\mathscr{Z}(\vZ^k)$ be the set of cluster points of $\{\vZ^k\}_{k\geq 0}$. Then, $\mathscr{Z}(\vZ^k) \subseteq \vS^*$ a.s.
\end{lemma}

\begin{proof}
We take several steps to complete the proof of this lemma.

%
{(i) Firstly, from Corollary~\ref{cor:summable}, we have $ Z^k-\bar { Z}^{k+1}\rightarrow 0$ a.s.. {Since $\| Z^k- Z^{k+1}\|_{\rm F}\leq {{\sqrt{\kappa}}\over nq_{\min}}\| Z^k-\bar Z^{k+1}\|_M$ (c.f.~\eqref{term2}),}
we have $ Z^k- Z^{k+1}\rightarrow 0\text{ a.s.}$ Then
from~\eqref{eqn:inconsist}, we have $\widehat  Z^k- Z^k\rightarrow 0\text{ a.s.}$

(ii) From Corollary \ref{cor:summable}, we have that $(\|\vZ^{k} - {\vZ}^* \|_U^2)_{k\geq 0}$ converges a.s., and so does $(\|\vZ^{k} - {\vZ}^* \|_U)_{k\geq 0}$. Hence, we have $\lim_{k\rightarrow \infty} \|\vZ^{k} - {\vZ}^* \|_U = \gamma$ a.s., where $\gamma$ is a $[0, +\infty )$-valued random variable. Hence, $(\|\vZ^{k} - {\vZ}^* \|_U)_{k\geq 0}$ must be bounded a.s., and so is $(\vZ^k)_{k\geq 0}$.

(iii) We claim that there exists $\widetilde{\Omega} \in \cF$ such that $P(\widetilde{\Omega}) = 1$ and, for every $\omega \in \widetilde{\Omega}$ and every ${\vZ}^* \in \vS^*$, $(\|\vZ^k(\omega) - {\vZ}^*\|_U)_{k\geq 0}$ converges.

The proof follows directly from \cite[Proposition 2.3 (iii)]{combettes2015stochastic}. It is worth noting that $\widetilde \Omega$ in the statement works for all ${\vZ}^* \in \vS^*$, namely, $\Omega$ does not depend on ${\vZ}^*$.

(iv) By~(i),  there exists $\widehat{\Omega} \in \cF$ such that $P(\widehat{\Omega} )=1$ and
$\label{lim-xk} Z^k(w)- Z^{k+1}(w)\rightarrow 0,\quad\forall w\in\widehat\Omega.$
For any $\omega\in\widehat{\Omega}$, let $(\vZ^{k_l}(\omega))_{l\ge0}$ be a convergent subsequence of $(\vZ^k(\omega))_{k\geq 0}$, i.e., $\vZ^{k_l}(\omega) \rightarrow \vZ$, where $\vZ^{k_l}(\omega)  = ( Z^{k_l}(\omega),  Z^{k_l - 1}(\omega)...,  Z^{k_l - \tau}(\omega) )$ and $\vZ = (\vu^0, ..., \vu^{\tau})$. Note that $\vZ^{k_l}(\omega) \rightarrow \vZ$ implies
$ Z^{k_l - j}(\omega) \rightarrow \vu^j,\, \forall j.$ 
Therefore, $\vu^i=\vu^j$, for any $i,j\in\{0,\cdots,\tau\}$ because $ Z^{k_l - i}(\omega) -  Z^{k_l - j}(\omega)\rightarrow 0$. Furthermore, observing $\eta>0$, we have
\begin{equation}\label{lim-xhat}
\begin{aligned}
&\ \textstyle \lim_{l\rightarrow \infty} \widehat  Z^{k_l} (\omega) - \cT \widehat  Z^{k_l}(\omega)\nnb
=&\ \lim_{l\to\infty}S\widehat  Z^{k_l}(\omega)=\lim_{l\to\infty}\frac{1}{\eta}( Z^{k_l}(\omega)-\bar{ Z}^{k_l+1}(\omega))=0.
\end{aligned}
\end{equation}
From the triangle inequality and the nonexpansiveness of $\cT$, it follows that
\begin{align*}
&\| Z^{k_l}(\omega) - \cT Z^{k_l}(\omega)\|_M \\
=& { \| Z^{k_l}\hspace{-0.5mm}(\omega) \hspace{-1mm}-\hspace{-1mm} \widehat  Z^{k_l}\hspace{-0.5mm} (\omega) \hspace{-1mm}+\hspace{-1mm} \widehat  Z^{k_l}\hspace{-0.5mm}(\omega) \hspace{-1mm}-\hspace{-1mm} \cT\widehat  Z^{k_l}\hspace{-0.5mm}(\omega) \hspace{-1mm}+\hspace{-1mm}
\cT\widehat  Z^{k_l}\hspace{-0.5mm}(\omega) \hspace{-1mm}-\hspace{-1mm} \cT Z^{k_l}\hspace{-0.5mm}(\omega)\|_M} \nnb
 \leq& \| Z^{k_l}(\omega) - \widehat  Z^{k_l}(\omega)\|_M + \| \widehat
  Z^{k_l}(\omega) - \cT\widehat  Z^{k_l}(\omega)\|_M \\
 &+ \| \cT\widehat
  Z^{k_l}(\omega) - \cT Z^{k_l}(\omega)\|_M \nnb
 \leq& 2 \| Z^{k_l}(\omega) - \widehat  Z^{k_l} (\omega)\|_M + \| \widehat
  Z^{k_l}(\omega) - \cT\widehat  Z^{k_l}(\omega)\|_M\\
 \leq&\textstyle 2\sum_{d\in J(k_l)}\| Z^{d}(\omega) \hspace{-1mm}-\hspace{-1mm}   Z^{d+1}
 (\omega)\|_M \hspace{-1mm}+\hspace{-1mm} \| \widehat  Z^{k_l}(\omega) \hspace{-1mm}-\hspace{-1mm} \cT\widehat  Z^{k_l}(\omega)\|_M.
\end{align*}
By~\eqref{lim-xk} and \eqref{lim-xhat}, we have from the above inequality that
$\lim_{l\rightarrow \infty}  Z^{k_l}(\omega) - \cT  Z^{k_l}(\omega) = 0.$
Now the demiclosedness principle~\cite[Theorem 4.17]{bauschke2011convex} implies
$\vu^0 \in \textbf{Fix }\cT$, and this implies the statement of the lemma.}
\hfill\end{proof}
From Lemma~\ref{lemma:convergence} and Opial's Lemma~\cite{opial1967weak}, we have the convergence theorem of the asynchronous algorithm.
\begin{theorem}\label{thm:async-convergence2}
Let $( Z^k)_{k\geq0}\subset \cH$ be the sequence generated by Algorithm~\ref{alg:asyn}, $\eta
\in (0, \eta_{\max}]$ for certain $\eta_{\max}$ satisfying \eqref{28abncmko}. If Assumptions \ref{assum:random}, \ref{ass-delay} and \ref{ass-4} hold,
then $( Z^k)_{k\geq 0}$ converges to a $\mathbf{Fix} \cT$-valued random variable a.s..
\end{theorem}
This theorem guarantees that, if we run the asynchronous algorithm~\ref{alg:asyn} with an arbitrary starting point $ Z^0$, then the sequence $\{ Z^k\}$ produced will converge to one of the solutions to problem~\eqref{prob:saddle} almost surely. From the upper bound of $\eta_{\max}$, we can see that we must relax the update more if maximum delay $\tau$ becomes larger, or if the matrix $M$ becomes more ill-conditioned. The relative computation speed of the slowest agent will also affect the relaxation parameter.

\vspace{-3mm}
\subsection{New step size rules}\label{sec:stepsize}
For both the synchronous and asynchronous algorithms, we require the step size $\alpha$ be less than ${2\rho_{\min}}/{L}$ (Assumption~\ref{assumfunc}). Both $\rho_{\min}$ and $L$ involve global properties across the network: $\rho_{\min}$ is related to the property of the matrix $V$ and $L$ is related to all the functions $s_i$. Unless they are known \emph{a priori}, we need to apply a consensus algorithm to obtain them.

In this section we introduce a step size $\alpha_i$ for each agent $i$ that depends only on its local properties. The following theorem is a consequence of combining results of~\cite[Lemma 10]{lorenz2015inertial} and the monotone operator theory.
\begin{theorem}
Let the iteration $ Z^{k+1}=\cT Z^k$ be defined as
\begin{equation}
\begin{dcases}
x^{i,k+1}\hspace{-0.8mm}=\hspace{-0.8mm}\prox_{\alpha_i r_i}\hspace{-1mm}\Big(\hspace{-1mm}\sum_{j\in \mathcal{N}_i}\hspace{-1mm} w_{ij}x^{j,k}\hspace{-1mm}-\hspace{-1mm}\alpha_i\nabla s_i(\hspace{-0.2mm}x^{i,k}\hspace{-0.2mm}) \hspace{-1mm}-\hspace{-1.5mm}\sum_{e\in \mathcal{E}_i}\hspace{-1mm} v_{ei}y^{e,k} \hspace{-1mm}\Big),\forall i,\nnb
y^{e,k+1}=y^{e,k}+ \big(v_{ei}x^{i,k} + v_{ej}x^{j,k} \big),~\forall e=(i,j)\in\cL_i,\nonumber
\end{dcases}
\end{equation}
with $\alpha_i=\frac{1}{L_i/\gamma+1-w_{ii}}$ for an arbitrary $0<\gamma<2$. Then under assumption~\ref{assumfunc}, the operator $\cT$ is an averaged operator.
\end{theorem}
Because this new operator $\cT$ is still an averaged operator, the convergence results, Corollary~\ref{synconv} and Theorem~\ref{thm:async-convergence2}, still hold.
%

{\subsection{Allowing unbounded delays}\label{sec:ubd}
Using the analytical tools developed in~\cite{Hannah2016unbounded}, we can further prove that our algorithm allows unbounded delays. The detailed proofs are omitted due to the page limit; interested readers are referred to~\cite{Hannah2016unbounded}.

We make the following assumption.
\begin{assumption}\label{assum:eqpr}
\begin{equation}
q_i=\frac{1}{n}, ~\forall i=1,2,...,n.
\end{equation}
\end{assumption} 
Assumption~\ref{assum:eqpr} is made for simplicity of formulas. Only $q_i \ge \epsilon >0$ for $i$ is needed.

\textbf{Stochastic unbounded delays}\\
Instead of Assumption~\ref{ass-delay}, we make the following assumptions on the delay vectors $\tau^k,\delta^k$.
\begin{assumption}\label{assum:stdelay}
The delay vectors $\{\tau^k\}_{k\geq 0}$ are i.i.d. random variables; so are the delay vectors $\{\delta^k\}_{k\geq 0}$. In addition, they are independent of the agents $i_1,\ldots,i_k,\ldots$
\end{assumption}
\begin{assumption}[Evenly old delays]\label{assum:evenly}
We assume that the delays are \emph{evenly old}, meaning that there exists an integer $B>0$, such that for all $k\geq 0$, all $i,j=1,...,n$, all $e,f=1,...,m$,
\begin{align}
|\tau^k_i-\tau^k_j|\leq B,\\
|\tau^k_i-\delta^k_e|\leq B,\\
|\delta^k_e-\delta^k_f|\leq B.
\end{align}
\end{assumption}
In Assumption~\ref{assum:evenly}, only the existence of B is required. It can be as long as it needs to be.
\begin{definition}\label{def:qt}
We define 
\begin{align}
\Delta^k:&=\max\{\tau^k_1,\ldots,\tau^k_n,\delta^k_1,\ldots,\delta^k_m\};\label{deltak}\\
\mathbf{P}_l:&=P(\Delta^k\geq l), \quad l = 0, 1, \ldots
\end{align}
\end{definition}
We have the following theorem.
\begin{theorem}\label{thm:stdelay}
Let $( Z^k)_{k\geq0}\subset \cH$ be the sequence generated by Algorithm~\ref{alg:asyn}, 
Let Assumptions \ref{assum:eqpr}, \ref{ass-4}, \ref{assum:stdelay} and \ref{assum:evenly} hold. 
Then $( Z^k)_{k\geq 0}$ converges to a $\mathbf{Fix} \cT$-valued random variable a.s. if either of the following holds:
\begin{enumerate}
\item $\sum_{l=1}^\infty(l\mathbf{P}_l)^{1/2}<\infty$; the step sizes for each agent $i$ are equal and satisfy\\ $0<\eta_i<\left(\kappa+\frac{1}{\sqrt{n}}\sum_{l=1}^\infty \mathbf{P}_l^{1/2}(l^{1/2}+l^{-1/2})\right)^{-1}$;
\item $\sum_{l=1}^\infty(\mathbf{P}_l)^{1/2}<\infty$, the $\eta_i$'s are equal and satisfy $0<\eta_i<\left(\kappa+\frac{2}{\sqrt{n}}\sum_{l=1}^\infty \mathbf{P}_l^{1/2}\right)^{-1}$.
\end{enumerate}
\end{theorem}
In both scenarios of Theorem~\ref{thm:stdelay}, $\mathbf{P}_l$ are required to decay fast when $l$ grows. Bigger step sizes can be taken when the delay distribution has thinner tails. In particular, bounded delays and delays following a geometric distribution are special cases of both scenarios.

\textbf{Unbounded deterministic delays}

In the following, we allow the delays to be arbitrary.
\begin{assumption}\label{assum:dtdelay}
The delay vectors $\tau^k,\delta^k$ are arbitrary, with $\liminf\Delta^k<\infty$ where $\Delta^k$ is defined in~\eqref{deltak}.
\end{assumption}
Assumption~\ref{assum:dtdelay} means that there exists an bound $\cB$, as large as it needs to be, such that for infinite many iterations, the delays are no larger than $\cB$.
\begin{definition}
Let $( Z^k)_{k\geq0}\subset \cH$ be the sequence generated by Algorithm~\ref{alg:asyn}. For all positive integers $T$, define $Q_T$ be a subsequence of $( Z^k)_{k\geq0}$ obtained by removing the iterates with $\Delta^k>T$.
\end{definition}
We have the following convergence result.
\begin{theorem}\label{thm:dtdelay}
Let $( Z^k)_{k\geq0}\subset \cH$ be the sequence generated by Algorithm~\ref{alg:asyn}. 
Let Assumptions \ref{assum:eqpr}, \ref{ass-4} and \ref{assum:dtdelay} hold. Fix $\gamma>0$, and let the step sizes $\eta_i$ vary by each iteration and satisfy
\begin{equation}
0<\eta_i^k<\left(\kappa+\frac{1}{\sqrt{n}}(1+\frac{1}{\gamma})+\frac{1}{2+\gamma}(\Delta^k+1)^{2+\gamma}) \right)^{-1}.
\end{equation}
Then for all $T\geq\liminf \Delta^k$, the sequence $Q_T$ converges to the same $\mathbf{Fix} \cT$-valued random variable a.s.
\end{theorem}
Theorem~\ref{thm:dtdelay} requires the step sizes to change according to the current delay and the convergence is applied to the subsequence with some iterates excluded as in Definition~\ref{def:qt}.
}

%% file: numerical.tex
In this part we simulate the performance of the proposed asynchronous algorithm \eqref{update_asyn}. We will compare it with its synchronous counterpart \eqref{syn-decentralized} (synchronous PG-EXTRA) as well as the proximal gradient descent algorithms (synchronous algorithm \eqref{dpg} and asynchronous algorithm \eqref{dpg-asyn}).



In the following experimental settings, we generate the network as follows. The location of the $n$ agents are randomly generated in a $30 \times 30$ area. If any two agents are within a distance of $15$ units, they are regarded as neighbors and one edge is assumed to connect them. Once the network topology is generated, the associated weighting matrix $W$ is produced according to the Metropolis-Hastings rule \cite{sayed2014adaptation}.

{In all simulations, we let each agent start a new round of update, with available neighboring variables which involves delays in general, instantly after the finish of last one. To guarantee each agent to be activated i.i.d,} the computation time of agent $i$ is sampled from an exponential distribution exp$(1/\mu_i)$. 
For agent $i$, $\mu_i$ is set as $2+|\bar{\mu}|$ where $\bar{\mu}$ follows the standard normal distribution  $N(0,1)$. The communication times is also simulated. The communication time between agents are independently  sampled from exp$(1/0.6)$. After the computation time is generated, the probability $q_i$ can be computed accordingly.

In all curves in the following figures, the relative error $\frac{\|X^{k}-X^*\|_F}{\|X^{0}-X^*\|_F}$ against time is plotted, where $X^*$ is the exact solution to \eqref{prob-general}.

%

%

\vspace{-4mm}
\subsection{Decentralized compressed sensing}\label{sec:dcs}
For \emph{decentralized compressed sensing}, each agent $i\in\{1,2,\cdots,n\}$ holds some measurements: $b_i=A_i x+e_i\in\RR^{m_i}$, where $A_i\in \RR^{m_i\times p}$ is a sensing matrix, $x\in\RR^p$ is the common unknown \emph{sparse} signal, and $e_i$ is i.i.d. Gaussian noise. The goal is to recover $x$. The number of  measurements $\sum_{i=1}^n m_i$ may be less than the number of unknowns $p$, so  we solve the  $\ell_1$-regularized least squares:
\begin{align}
\Min_x~\textstyle{\frac{1}{n}\sum_{i=1}^n} s_i(x)+r_i(x),\label{prob:cs}
\end{align}
where $s_i(x)=\frac{1}{2}\|A_i x-b_i\|_2^2,\quad r_i(x)=\theta_i\|x\|_1,$
and $\theta_i$ is the regularization parameter with agent $i$.

The tested network has 10 nodes and 14 edges. We set $m_i=3$ for $i=1,\cdots,10$ and $p=50$. The entries of $A_i,e_i$ are  independently sampled from the standard normal distribution $N(0,1)$, and $A_i$ is normalized so that $\|A_i\|_2=1$. The signal $x$ is  generated randomly with $20\%$ nonzero elements. We set the regularization parameters $\theta_i=0.01$.

The step sizes of all the four algorithms are tuned by hand and are nearly optimal. The step size $\alpha$ for both the primal synchronous algorithm \eqref{dpg} and the asynchronous algorithm \eqref{dpg-asyn} are set to be $0.05$. {This choice is a compromise between convergence speed and accuracy.} The relaxation parameters for the asynchronous algorithm \eqref{dpg-asyn}, are chosen to be $\eta_i={0.036}/{q_i}$. The step size $\alpha$ for both synchronous PG-EXTRA Algorithm \ref{alg:syn} and asynchronous Algorithm~\ref{alg:asyn} are set to be $1$. The relaxation parameters for asynchronous Algorithm~\ref{alg:asyn} are chosen to be $\eta_i={0.0288}/{q_i}$.
From Fig.~\ref{fig:CS} we can see that asynchronous primal algorithm \eqref{dpg-asyn} is faster than its synchronous version \eqref{dpg}, but both algorithms are far slower than synchronous PG-EXTRA Algorithm \ref{alg:syn} and asynchronous Algorithm~\ref{alg:asyn}. The latter two algorithms exhibit linear convergence and Algorithm~\ref{alg:asyn}  converges significantly faster. {Within the same period (roughtly 2760ms), the two asynchronous algorithm finishes 21 times as many rounds of computation and communication as the synchronous counterparts, due to the elimination of waiting time.
	
To better illustrate the reason why the asynchronous Algorithm \ref{alg:asyn} is much faster than the synchronous Algorithm \ref{alg:syn}, we list the computation and communication time during the first iteration in Tables \ref{table-computation} and \ref{table-communication}, respectively. In Table \ref{table-communication} each edge corresponds to two rounds of communication time: communications from agent $i$ to $j$ and from $j$ to $i$. From Table \ref{table-computation} it is observed that the longest computation time is $1.152$ ms. From Table \ref{table-communication} it is observed that the longest communication time is $4.592$ ms. Therefore, the duration for the first iteration in the synchronous Algorithm \ref{alg:syn} is $1.152 + 4.592 = 5.744$ ms. However, in the asynchornous Algorithm \ref{alg:asyn}, each agent will start a new iteration right after the finish of its previous update. The average duration for the first iteration is $\sum_{i=1}^{10}t_i/10 = 0.443$ ms. Therefore, during the first iteration the asynchronous Algorithm \ref{alg:asyn} is $5.744/0.443\simeq 13$ times faster than the synchronous Algorithm \ref{alg:syn}. The data listed in Tables \ref{table-computation} and \ref{table-communication} illustrates the necessity to remove the idle time appearing in the synchronous algorithm.

\begin{table}
	\setlength{\belowcaptionskip}{-4mm}
	\begin{center}
		\begin{tabular}{|c|c|c|c|c|c| } 
			\hline
			agent & 1 & 2 & 3 &4 & 5 \\
			\hline
			time (ms) & 0.497 & 0.033 & 0.944 & 0.551 & 1.152 \\
			\hline \hline
			agent & 6 & 7 & 8 &9 & 10 \\
			\hline
			time (ms) & 0.072 &  0.112 &  0.996 &  0.049  & 0.025 \\
			\hline
		\end{tabular}
		\caption{\small Sampled computation time in the $1$st iteration.}
		\label{table-computation}
	\end{center}
\end{table}



\begin{table}
		\setlength{\belowcaptionskip}{-4mm}
	\begin{center}
		\begin{tabular}{|c|c|c|c|c|c|c|c|c} 
			\hline
			edge & (1,2) & (1,10) & (1,8) &(2,3) & (2,5) & (2,8) & (2,9) \\
			\hline
			time(ms) & 0.489 & 1.425 & 0.024 & 1.191 & 2.862 & 2.140 & 0.091 \\
			\hline \hline
			edge & (3,6) & (4,6) & (4,7) & (4,9)  & (6,7) & (7,8) & (8,10)\\
			\hline 
			time(ms) & 1.429 & 0.018 & 2.359 & 2.233 & 0.003 & 1.952 & 2.412 \\
			\hline \hline
			edge & (2,1) & (10,1) & (8,1) &(3,2) & (5,2) & (8,2) & (9,2)\\
			\hline
			time (ms)  & 2.762 & 1.165 & 1.672 & 1.828 & 0.569 & 4.592 & 0.617 \\
			\hline \hline
			edge  & (6,3) & (6,4) & (7,4) & (9,4) & (7,6) & (8,7) & (10,8) \\
			\hline 
			time (ms)  & 0.385 &0.887 & 1.152 & 0.744 & 2.716  &  0.649&  3.031 \\
			\hline
		\end{tabular}
		\caption{\small Sampled communication time in the $1$st iteration. \vspace{-0mm}}
		\label{table-communication}
	\end{center}
\end{table}

}

Since both synchronous and asynchronous proximal-gradient-descent-type algorithms are much slower compared to the synchronous PG-EXTRA Algorithm~\ref{alg:syn} and asynchronous Algorithm~\ref{alg:asyn}, in the following two simulations we just show convergence performance of PG-EXTRA Algorithm~\ref{alg:syn}  and asynchronous Algorithm~\ref{alg:asyn}.

\begin{figure}
	\setlength{\belowcaptionskip}{-3mm}
\centering
\includegraphics[scale=0.4]{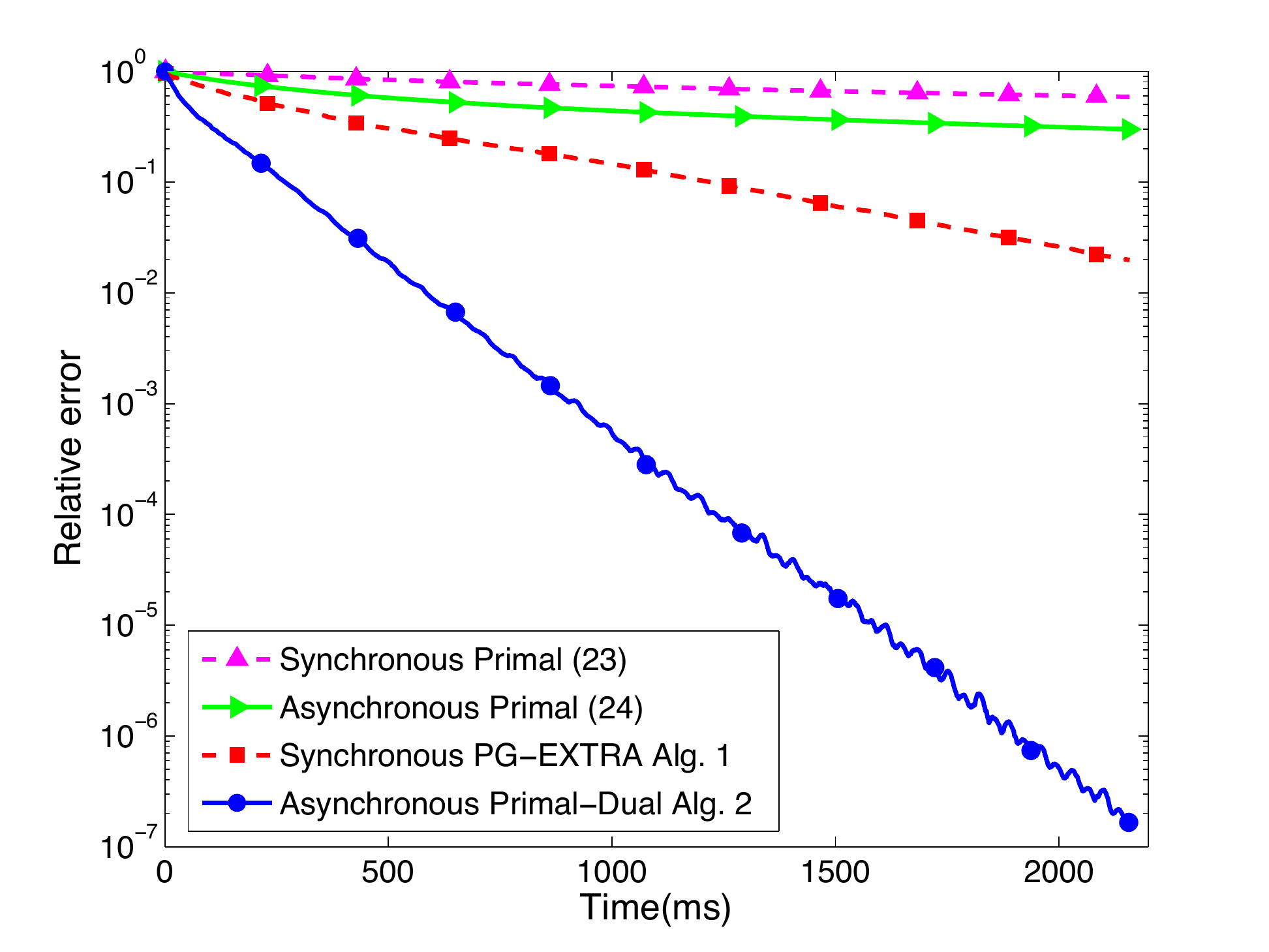}
\caption{\small Convergence comparison between synchronous algorithm \protect\eqref{dpg}, asynchronous algorithm \protect\eqref{dpg-asyn}, synchronous PG-EXTRA Algorithm \protect\ref{alg:syn} and the asynchronous Algorithm \protect\ref{alg:asyn} for compressed sensing.}
\label{fig:CS}
\end{figure}
\vspace{-3mm}
\subsection{Decentralized sparse logistic regression}
In this subsection the tested problem is {\em decentralized sparse logistic regression}. Each agent $i\in\{1,2,\cdots,n\}$ holds local data samples $\{h^i_j, d^i_j\}_{j=1}^{m_i}$, where the supscript $i$ indiates the agent index and subscript $j$ indicates the data index. $h_j^i\in \RR^p$ is a feature vector and $d_j^i \in \{+1, -1\}$ is the corresponding label. $m_i$ is the number of local samples kept by agent $i$. All agents in the network will cooperatively solve the sparse logistic regression problem
\eq{\label{sparse-lr}
\textstyle{\Min_{x\in \RR^p}\quad  \frac{1}{n} \sum_{i=1}^{n} [s_i(x) + r_i(x)],}
}
where $\hspace{-0.5mm}s_i(x) \hspace{-1mm}=\hspace{-1mm} \frac{1}{m_i}\hspace{-1mm} \sum_{j=1}^{m_{i}} \hspace{-1mm}\ln \hspace{-0.4mm}\left( \hspace{-0.3mm}1 \hspace{-0.8mm}+\hspace{-0.8mm} \exp(-d_j^i (h_j^i)^\top\hspace{-0.8mm} x  \hspace{-0.4mm}) \right)$, $r_i(x) \hspace{-1mm} = \hspace{-1mm}\theta_{i} \|x\|_1$.

In the simulation, we set $n=10$, $p=50$, and $m_{i}=3$ for all $i$. For local data samples $\{h_j^i, d_j^i\}_{j=1}^{m_{i}}$ at agent $i$, each $h_j^i$ is generated from the standard normal distribution $N(0,1)$. To generate $d_{j}^i$, we first generate a random vector $x^o\in \RR^{p}$ with $80\%$ entries being zeros. Next, we generate $d_j^i$ from a uniform distribution $U(0,1)$. If $d_j^i \le 1/[1+\exp(-(h_j^i)^\top x^o)]$ then $d_j^i$ is set as $+1$; otherwise $d_j^i$ is set as $-1$. We set the regularization parameters $\theta_i=0.1$.

%

The step sizes of both algorithms are tuned by hand and are nearly optimal. The step sizes $\alpha$ for both synchronous PG-EXTRA Algorithm \ref{alg:syn} and asynchronous Algorithm~\ref{alg:asyn} are set to be $0.4$. The relaxation parameters for asynchronous Algorithm~\ref{alg:asyn} are chosen to be $\eta_i=\frac{0.0224}{q_i}$.
From Fig.~\ref{fig:logistic} we can see that both algorithms exhibit convergence almost linearly and that Algorithm~\ref{alg:asyn}  converges significantly faster. In our experiments, within the same period, the asynchronous algorithm finishes 21 times as many rounds of computation and communication as that performed by the synchronous algorithm, due to the elimination of waiting time.

\begin{figure}
	\setlength{\belowcaptionskip}{-5mm}
	\centering
\includegraphics[scale=0.4]{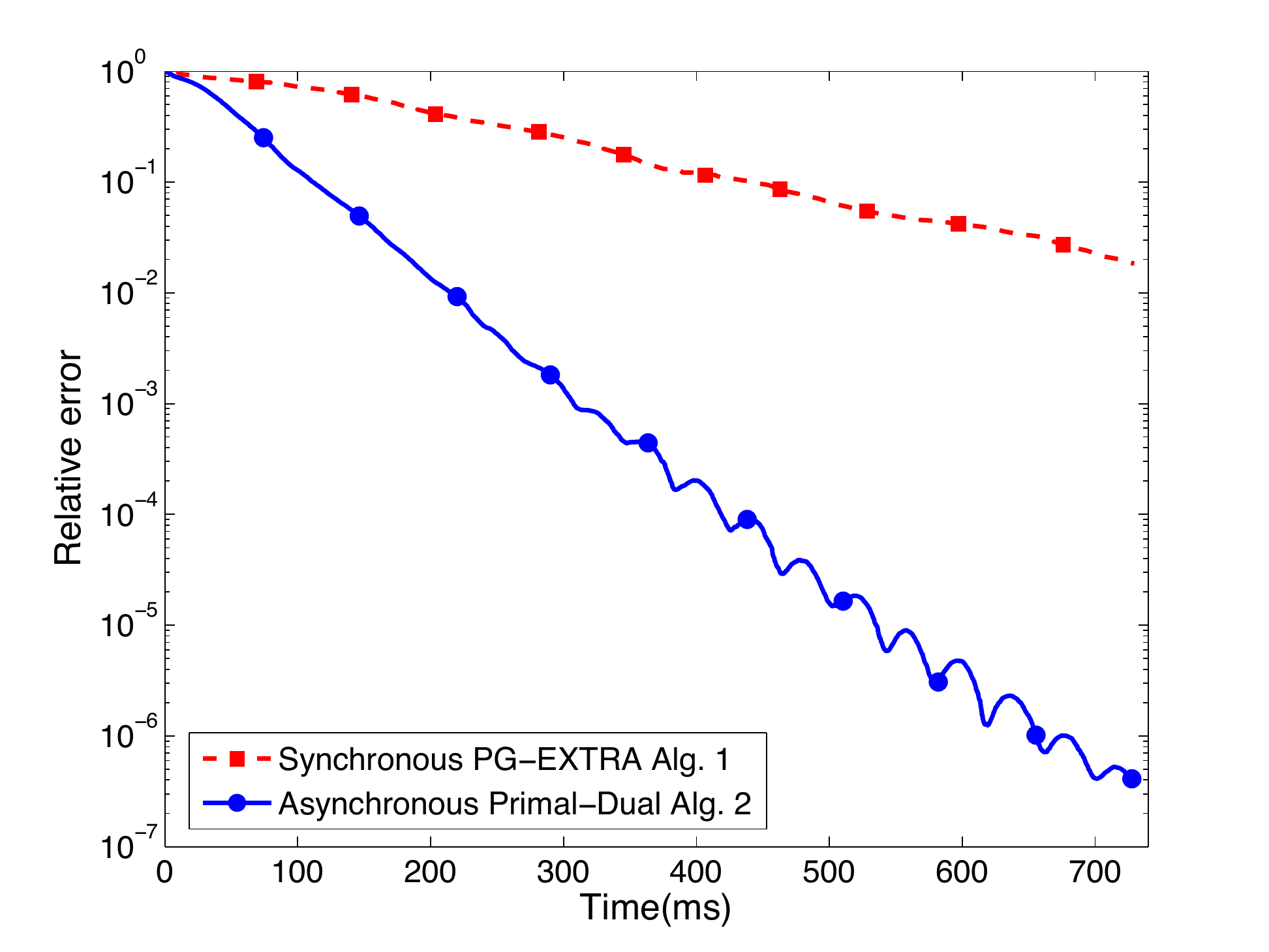}
\caption{\small Convergence comparison between synchronous PG-EXTRA and the asynchronous Algorithm \protect\ref{alg:asyn} for sparse logistic regression.}
\label{fig:logistic}
\end{figure}

\subsection{Decentralized low-rank matrix completion}
Consider a low-rank matrix
$A = [A_1, \cdots, A_n] \in \RR^{N \times K}$
of rank $r \ll \min\{N,K\}$. In a network, each agent $i$ observes some entries of $A_i \in \RR^{N\times K_i}$, $\sum_{i=1}^{n}K_i = K$. The set of observations is $\Omega = \cup_{i=1}^n \Omega_i$. To recover the unknown entries of $A$, we introduce a {\em public matrix} $X \in \RR^{N\times r}$, which is known to all agents, and a {\em private matrix}
$Y = [Y^1, \cdots, Y^n] \in \RR^{r \times K},$ 
where each $Y^i\in \RR^{N\times K_i}$ corresponds to $A_i$ and is held by agent $i$. The supscript $i$ indicates the agent index. We reconstruct $A = XY$, which is at most rank $r$, by recovering $X$ and $Y$ in a decentralized fashion. The problem is formulated as follows (see \cite{ling2012decentralized} for reference).
\eq{
	\label{matrix fac}
	\Min_{X, \{Y_i\}_{i=1}^n, \{Z_i\}_{i=1}^n} \quad& \textstyle{\frac{1}{2}\sum_{i=1}^{n}} \|XY^i - Z^i\|_F^2, \nnb
		\St \quad & (Z^i)_{ab} =  (A_i)_{ab}, \forall (a,b)\in \Omega_i,
}
where $Z^i\in \RR^{N\times K_i}$ is an auxiliary matrix, and $(Z^i)_{ab}$ is the $(a,b)$-th element of $A_i$.
\subsubsection{Synchronous algorithm}
\cite{ling2012decentralized} proposes a decentralized algorithm to solve Problem \eqref{matrix fac}. Let each agent hold $X^i$ as a local copy of the public matrix $X$, the algorithm is:

\vspace{2mm}
\noindent \textit{Step 1: Initialization.} Agent $i$ initializes $X^{i,0}$ and $Y^{i,0}$ as random matrices, $Z^{i,0}$ is also initialized as a random matrix with $(Z^{i,0})_{ab}=(A_i)_{ab}$ for any $(a,b)\in \Omega_i$.

\noindent \textit{Step 2: Update of $X^i$.} Each agent $i$ updates $X^{i,k+1}$ by solving the following average consensus problem:
\eq{\label{ave-cns}
\Min_{\{X^i\}_{i=1}^n}\quad& \textstyle{\frac{1}{2}\sum_{i=1}^{n} \|X^i - Z^{i,k} (Y^{i,k})^\top\|_F^2,} \nnb
\St \quad& \textstyle{X^1 = X^2 = \cdots = X^n}.
}

\noindent \textit{Step 3: Update of $Y^i$.} Each agent $i$ updates
$$ Y^{i,k+1} = \big[ (X^{i,k+1})^\top X^{i,k+1} \big]^{-1} (X^{i,k+1})^\top Z^{i,k}.$$

\noindent \textit{Step 4: Update of $Z_i$.} Each agent $i$ updates
\eq{
Z^{i,k+1} =&  X^{i,k+1} Y^{i,k+1} + P_{\Omega_i} \big(A_i - X^{i,k+1} Y^{i,k+1} \big),
}
where $P_{\Omega}(\cdot)$ is defined as follows: {\color{black}for any matrix $A$, if $(a,b)\in \Omega$, then $[P_{\Omega}(A)]_{ab} = A_{ab}$; otherwise $[P_{\Omega}(A)]_{ab} = 0$.}


In problem \eqref{ave-cns} appearing at Step 2, we can let $r_i(X^i)=\frac{1}{2}\|X^i - Z^{i,k} (Y^{i,k})^\top\|_F^2$ and $s_i(X_i)=0$ and hence problem \eqref{ave-cns} falls into the general form of problem \eqref{prob-general}, for which we can apply the synchronous PG-EXTRA Algorithm \ref{alg:syn} to solve it. We introduce matrices $\{Q^e,\forall e\in\cE\}$ as dual variables. Instead of solving \eqref{ave-cns} exactly, we just run \eqref{syn-decentralized} once for each iteration $k$. Therefore, {\em Step 2} becomes

\noindent \textit{Step 2$^\prime$: Update of $X^i$.} Each agent $i$ updates $X^{i,k+1}$ by
\eq{\label{update step 2'}
		X^{i,k+1}\hspace{-0.8mm} &=\hspace{-0.5mm}\frac{1}{\alpha\hspace{-0.8mm}+\hspace{-0.8mm}1} \hspace{-1mm} \Big( \hspace{-1mm}\sum_{j\in \mathcal{N}_i} \hspace{-1.5mm} w_{ij}X^{j,k}\hspace{-0.8mm}-\hspace{-1.5mm} \sum_{e\in \mathcal{E}_i} \hspace{-1mm}v_{ei} Q^{e,k} \hspace{-0.8mm}+\hspace{-0.8mm} \alpha  Z^{i,k} (Y^{i,k})^\top \hspace{-1mm}\Big), \\
		Q^{e,k+1}&= Q^{e,k}\hspace{-0.8mm}+\hspace{-0.8mm} \big(v_{ei} X^{i,k} \hspace{-0.8mm}+\hspace{-0.8mm} v_{ej} X^{j,k} \big),~\forall e=(i,j)\in\cL_i.
}
\subsubsection{Asynchronous algorithm}
For the asynchronous algorithm, agent $i$, when activated, updates $X_i$ and $\{Q^e,\forall e=(i,j)\in\cL_i\}$ per {\em Step 2'} using the neighboring information it has available, then performs {\em Step 3} and {\em Step 4} to update its private matrices $Y_i$ and $Z_i$, and sends out the updated $X_i$ and $Q^e$ to neighbors.

We test our algorithms on a network with $20$ nodes and $41$ edges. In the simulation, we generate a rank $4$ matrix $A\in \RR^{40 \times 140}$, hence each $A_i \in \RR^{40 \times 7}$. To generate $A$, we first produce $E\in \RR^{40 \times 4} \sim N(0,1)$ and $F\in \RR^{140 \times 4} \sim N(0,1)$. We also produce a diagonal matrix $D\in \RR^{4\times 4} \sim N(0,1)$. With $E,F$ and $D$, we let $A = E D F^\top$. A total of $80\%$ entries of $A$ are known. To generate $\Omega$, we first sample its entries independently from the uniform distribution $U(0,1)$. If any entry of $\Omega$ is less than $0.8$, it will be set as $1$, which indicates this entry is known; otherwise it will be set as $0$.
%

We run the synchronous PG-EXTRA Algorithm~\ref{alg:syn} and the proposed asynchronous primal-dual  Algorithm~\ref{alg:asyn} and plot the relative error $\frac{\|Z^{k}-A\|_F}{\|Z^0-A\|_F}$ against time, as depicted in Fig.~\ref{fig:matrix_comp}. The step sizes of both algorithms are chosen to be $\alpha=0.1$. The relaxation parameters for asynchronous Algorithm~\ref{alg:asyn} are chosen to be $\eta_i=\frac{0.0102}{q_i}$. The performance of the algorithms are similar to the previous experiments.
In this $20$-node network, within the same period, the asynchronous algorithm finishes $29$ times as many rounds of computation and communication as those finished by synchronous Algorithm \ref{alg:syn}.
\begin{figure}
	\setlength{\belowcaptionskip}{-5mm}
	\centering
\includegraphics[scale=0.4]{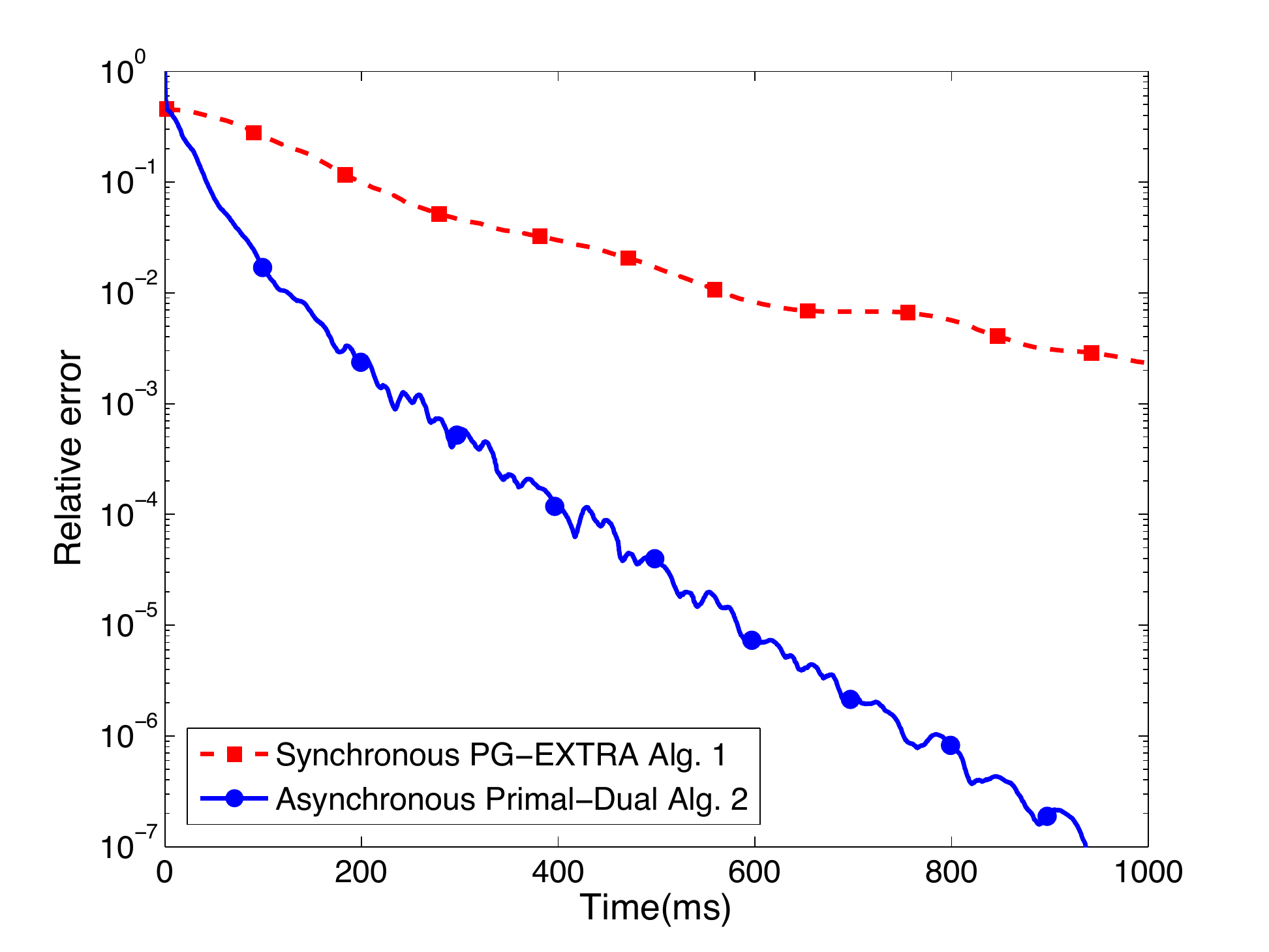}
\caption{\small Convergence comparison between synchronous PG-EXTRA and the asynchronous Algorithm \protect\ref{alg:asyn} for matrix completion.}
\label{fig:matrix_comp}
\end{figure}

{
\subsection{Decentralized geometric median}
In the literature there are limited asynchronous algorithms proposed to solve problems with composite cost function as in Problem \eqref{prob-general}. When there is no differentiable term, i.e.,  $s_i(x)=0$, the most closely related algorithm to the proposed asynchronous primal-dual Alg.~\ref{alg:asyn} is the asynchronous ADMM \cite{Peng_2015_AROCK}. In this subsection, we compare these two algorithms when solving the geometric-median problem:
\eq{
\min_{x\in \RR^p}\quad \frac{1}{n}\sum_{i=1}^{n}\|x - b_i \|_2,
}
where $\{b_i\}_{i=1}^n$ are given constants. Computing decentralized geometric medians have various interesting applications, see \cite{shi2015proximal} for more detail. In this simulation, we set $n=11$, $p=4$, and each $b_i$ is generated from the Gaussian distribution $N(0, \Lambda)$, where $\Lambda$ is a diagonal matrix with each entry follows uniform distribution $U(0,10)$. The parameters of both algorithms are hand-optimized. The augmented coefficient in the asynchronous ADMM is 0.3, and the step-size in Alg. \ref{alg:asyn} is $1$. The relaxation parameter for both algorithms is set as $0.4$. It is observed in Fig. \ref{fig:mp_10nodes} that Alg. \ref{alg:asyn} is faster than the asynchronous ADMM.
\begin{figure}
	\setlength{\belowcaptionskip}{-5mm}
	\centering
	\includegraphics[scale=0.4]{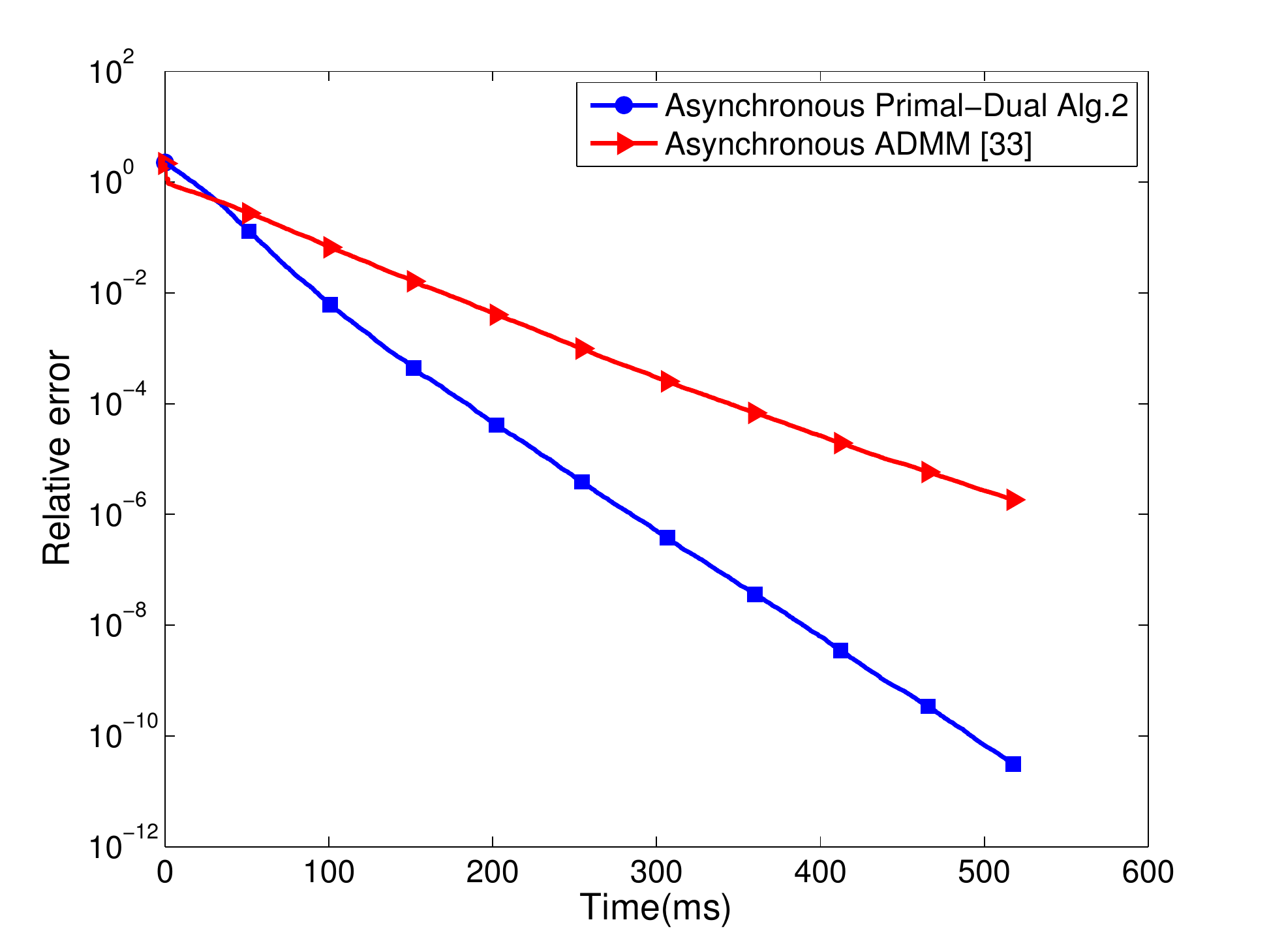}
	\caption{\small Convergence comparison between asynchronous ADMM and the proposed Alg. \protect\ref{alg:asyn} for decentralized geometric median problem.}
	\label{fig:mp_10nodes}
\end{figure}

}

%% file: conclusion.tex
\vspace{-2mm}

\section{Conclusions}
This paper developed an asynchronous, decentralized algorithm for concensus optimization. The agents in this algorithm can compute and communicate in an uncoordinated fashion; local variables are updated with possibly out-of-date information. Mathematically, the developed algorithm extends the existing algorithm PG-EXTRA by adding an explicit dual variable for each edge and to take asynchronous steps. The convergence of our algorithm is established under certain statistical assumptions. Although not all assumptions are satisfied in practice, the algorithm is practical and efficient. In particular, step size parameters are fixed and depend only on local information.

Numerical simulations were performed on both convex and nonconvex problems, and synchronous and asynchronous algorithms were compared. In addition, we introduced an asynchronous algorithm without dual variables by extending an existing algorithm and included it in the simulation. All simulation results clearly show the advantages of the developed asynchronous primal-dual algorithm.

%
%
%
%